\documentclass[11pt,reqno,tbtags]{amsart}

\usepackage{amsthm,amssymb,amsfonts,amsmath}
\usepackage{amscd} 
\usepackage{mathrsfs}
\usepackage[numbers, sort&compress]{natbib}
\usepackage{subfigure, graphicx}
\usepackage{vmargin}
\usepackage{setspace}
\usepackage{paralist}
\usepackage[pagebackref=true]{hyperref}
\usepackage{color}
\usepackage[normalem]{ulem}
\usepackage{stmaryrd} 
\usepackage[refpage,noprefix,intoc]{nomencl} 

\providecommand{\R}{}
\providecommand{\Z}{}

\renewcommand{\R}{\mathbb{R}}
\renewcommand{\Z}{\mathbb{Z}}

\newcommand{\E}[1]{{\mathbf E}\left[#1\right]}
\newcommand{\e}{{\mathbf E}}
\newcommand{\V}[1]{{\mathbf{Var}}\left\{#1\right\}}

\newcommand{\p}[1]{{\mathbf P}\left\{#1\right\}}
\newcommand{\psub}[2]{{\mathbf P}_{#1}\left\{#2\right\}}

\newcommand{\I}[1]{{\mathbf 1}_{[#1]}}
\newcommand{\set}[1]{\left\{ #1 \right\}}
\newcommand{\Cprob}[2]{\mathbf{P}\set{\left. #1 \; \right| \; #2}}
\newcommand{\probC}[2]{\mathbf{P}\set{#1 \; \left|  \; #2 \right. }}

 \newcommand{\bag}{\begin{align}}
\newcommand{\bags}{\begin{align*}}
\newcommand{\eag}{\end{align*}}
\newcommand{\eags}{\end{align*}}

\newtheorem{thm}{Theorem}[section]
\newtheorem{lem}[thm]{Lemma}
\newtheorem{prop}[thm]{Proposition}
\newtheorem{cor}[thm]{Corollary}

\newtheorem{fact}[thm]{Fact}

\newcommand\cF{\mathcal F}

\newcommand{\bM}{\mathbf{M}}

\newcommand{\bX}{\mathbf{X}}

\newcommand{\pran}[1]{\left(#1\right)}

\providecommand{\eps}{}
\renewcommand{\eps}{\epsilon}
\providecommand{\ora}[1]{}
\renewcommand{\ora}[1]{\overrightarrow{#1}}

\newcommand{\eqdist}{\ensuremath{\stackrel{\mathrm{d}}{=}}}

\newcommand{\pst}{\ensuremath{\preceq_{\mathrm{st}}}}

\hypersetup{
    bookmarks=true,         
    unicode=false,          
    pdftoolbar=true,        
    pdfmenubar=true,        
    pdffitwindow=true,      
    pdftitle={My title},    
    pdfauthor={Author},     
    pdfsubject={Subject},   
    pdfnewwindow=true,      
    pdfkeywords={keywords}, 
    colorlinks=true,       
    linkcolor=blue,          
    citecolor=blue,        
    filecolor=blue,      
    urlcolor=blue           
}

\reversemarginpar
\definecolor{clou}{rgb}{0.8,0.25,0.5125}

\newcommand{\bx}{\mathbf{x}}
\newcommand{\bm}{\mathbf{m}}

\begin{document}

\title{The front location in BBM with decay of mass} 
\author{Louigi Addario-Berry}
\author{Sarah Penington}
\address{Department of Mathematics and Statistics, McGill University, 805 Sherbrooke Street West, 
		Montr\'eal, Qu\'ebec, H3A 2K6, Canada}
\address{Department of Statistics, University of Oxford, 24-29 St Giles', Oxford, OX1 3LB, UK}

\email{louigi@problab.ca}
\email{sarah.penington@sjc.ox.ac.uk}
\date{December 9, 2015} 

\begin{abstract} 
We augment standard branching Brownian motion by adding a competitive interaction between nearby particles. Informally, when particles are in competition, the local resources are insufficient to cover the energetic cost of motion, so the particles' masses decay. In standard BBM, we may define the {\em front displacement} at time $t$ as the greatest distance of a particle from the origin. For the model with masses, it makes sense to instead define the front displacement as the distance at which the local mass density drops from $\Theta(1)$ to $o(1)$. We show that one can find arbitrarily large times $t$ for which this occurs at a distance $\Theta(t^{1/3})$ behind the front displacement for standard BBM. 
\end{abstract}

\maketitle

\section{Introduction}\label{sec:intro}
In this work, we propose a mathematical model of competition for resources within a single species, in a growing, spatially structured population, and provide an initial study of the front location in this new setting.
The model is essentially standard one-dimensional branching Brownian motion (BBM), augmented with a destructive, local interaction between particles. We first briefly recall BBM: start from a single particle at a point in $\mathbb{R}$, endowed with an Exp$(1)$ ``branching clock''. The particle moves according to Brownian motion; when its clock rings, it splits in two (branches). The new particles receive independent Exp$(1)$ clocks, and move independently (according to Brownian motion) starting from where the first particle splits, until their own clocks ring and they in turn split, {\em et cetera}. 

Write $n(t)$ for the total number of particles at time $t$, and $\bX(t)= (X_i(t),1 \le i \le n(t))$ for the locations of such particles. 
We assume the particles are listed in a way that makes the vector $\bX(t)$ exchangeable; one possible formalism is via the Ulam-Harris tree, with particles listed lexicographically acording to their label in the tree. We refer the reader to \cite{harris14} for more details on such matters; but many different references are possible. 
We also write $N(t,x) = \{i:X_i(t) \ge x\}$ for the indices of particles with position greater than $x$ at time $t$. 

We sometimes write $(X_i(t),i \ge 1)$, ignoring the fact that $\bX(t)$ has finite length, for convenience. 
We adopt the convention that $X_k(t)=\partial$ for $k > n(t)$ (so $\partial$ is where new babies come from). We refer to ``the particle $X_i(t)$'' as shorthand for ``the particle with position $X_i(t)$ at time $t$''; this is unambiguous at $\mathrm{Leb}$-a.e.\ time $t$. We write $\mathbf{P}$ for the probability measure under which $(\bX(t),t \ge 0)$ has the law of one dimensional BBM with initial individual at $0$, $\mathbf{E}$ for the corresponding expectation, and $(\cF_t,t \ge 0)$ for the filtration generated by the process. 

We now add destructive interaction as follows. Informally, imagine that the particles are, say, amoeba. Motion has an energetic cost, but for a single particle in isolation, this cost is exactly accounted for by the resources available in the environment. When particles are nearby, however -- at distance less than one, say -- they must share resources; in this case individuals do not consume enough to meet their energy expenditure, and their mass decreases. Finally, larger (more massive) individuals consume resources at a greater rate.

Formally, we define a vector $\bM(t)=(M_i(t),i \ge 0)$, and call $M_i(t)$ the {\em mass} of particle $X_i(t)$. 
By convention, if $X_i(t)=\partial$ then $M_i(t)=0$. Write 
\[
\zeta(t,x) = \sum_{\{i:|X_i(t)-x| \in (0,1)\}} M_i(s)
\]
for the total mass of particles within distance one of $x$ at time $t$, excluding any particles at position $x$. 
Then at time $t$, $M_i(t)$ decays at rate $\zeta(t,X_i(t))$. In other words, 
$\mathrm{d} M_i(t) = -M_i(t)\cdot \zeta(t,X_i(t))\mathrm{d}t$, so 
\[
M_i(t) = \exp(-\int_0^t \zeta(s,X_i(s)) \mathrm{d}s)\, .
\]
This should be viewed as defining $(\bM(t),t \ge 0)$ to be the solution of a system of differential equations; the definition makes sense since the system has a unique solution $\mathbf{P}$-almost surely. Furthermore, the process $(\bM(t),t \ge 0)$ is clearly $\cF_t$-adapted. 

For later use, write $(X_{i,t}(s),0 \le s \le t)$ for the ancestral path leading to $X_i(t)$, and let $\sigma_i(t)$ be the final branching time along this path. Also, let $j_{i,t}(s)$ be the index of $X_{i,t}(s)$ among the time-$s$ population, so that $X_{i,t}(s) = X_{j_{i,t}(s)}(s)$. 
We also write $M_{i,t}(s)$ for the mass of the ancestor of $X_i(t)$ at time $s$ (so $M_{i,t}(s)=M_{j_{i,t}(s)}(s)$).

Note that along any given trajectory, mass decreases: $(M_{i,t}(s),0 \le s \le t)$ is decreasing in $s$ for each $1 \le i \le n(t)$. Mass enters the system through branching events, since each ``child particle'' inherits the mass of its parent. This is obviously physically unrealistic in some settings (e.g.\ for amoebae) but may be more realistic in others (e.g.\ in nuclear physics). 

Rather than viewing $M_i(t)$ as a mass, a perspective suggested by a referee is to view $(M_{i,t}(s),0 \le s \le t)$ as recording information about the local density of the environment observed along the ancestral trajectory of the particle $X_i(t)$. The interaction between the dynamics of $\bX(t)$ and $\bM(t)$ makes this point of view slightly complicated to interpret, but here is one possibility.  Imagine adding destructive interaction to a BBM, as follows: whenever two different particles are at distance less than $1$, each kills the other at rate one. Record such a killing event as a mark at the appropriate location of the BBM family tree. Particles with a mark on their ancestral trajectory are {\em ghosts}, which continue to move and reproduce as before, but can no longer kill other particles.  Given the BBM but not the marks, one may ask for the conditional survival probabilities $p_i(t) = \probC{X_i(t)\mbox{ is alive}}{\cF_t}$ of the particles. The vector $\bM(t)$ is a ``linearized'' version of the vector of these survival probabilities.


\subsection{Main result}
Write 
\[
d(t,m) = \min\{x > 0: \zeta(t,x) < m\}, \quad D(t,m) = \max\{x: \zeta(t,x) > m\},
\]
for the leftmost (positive) location at which the total mass of nearby particles falls below $m$, and
the rightmost location at which it exceeds $m$, respectively. 
We prove the following theorem. 
\begin{thm}\label{thm:main}
Write 
$c^*=3^{4/3}\pi^{2/3}/2^{7/6}$. 
Then almost surely, for all $m < 1$, 
\[
\limsup_{t \to \infty} \frac{\sqrt{2}t-d(t,m)}{t^{1/3}} \ge c^* \quad\mbox{and}\quad 
\liminf_{t \to \infty} \frac{\sqrt{2}t-D(t,m)}{t^{1/3}} \le c^*
\]
\end{thm}
A well-known result of \citet{bramson78} states that the rightmost particle location $\max_{i \ge 1} X_{i}(t)$ has median $\mathrm{med}(t)$ satisfying 
\[
\mathrm{med}(t) = \sqrt{2}t - \frac{3}{2^{3/2}} \log t + O(1). 
\]
Furthermore, it turns out \cite{hu09min} that $|\max_{i \ge 1} X_i(t) - \mathrm{med}(t)|$ is almost surely $O(\log t)$, in that $\limsup_{t \to \infty} |\max_{i \ge 1} X_i(t) - \mathrm{med}(t)|/\log t$ is a.s.\ finite. 
In view of this, the theorem states that (1) there are (large) times $t$ at which the first low-density region lags at least distance $c^* t^{1/3}+o(t^{1/3})$ behind the rightmost particle, and (2) there are also (potentially different, large) times $t$ at which there is some high-density region within distance $c^* t^{1/3}+o(t^{1/3})$ of the rightmost particle. 

We believe that in fact almost surely, for all $m < 1$,  
\[
\lim_{t \to \infty} \frac{\sqrt{2}t-d(t,m)}{t^{1/3}} = c^* = \lim_{t \to \infty} \frac{\sqrt{2}t-D(t,m)}{t^{1/3}}\, .
\]
If this is correct, then the front could equivalently be defined as, e.g., a median of $D(t,m)$ or $d(t,m)$ -- or any other fixed quantile of one of these random variables. We provide some justification for our belief in Section~\ref{sec:conc}. That section also contains a few open questions about the model and a discussion of various generalizations of our results (some straightforward, some conjectural), as well as describing variants of the model which have thus far resisted analysis. 

\section{Proof sketch}
Here comes an outline of the key tools in our argument. The first is technical but important and also, we believe, provides important intuition when making heuristic predictions about the behaviour of the process. The remainder gives a fairly detailed overview of the proof. 

\medskip \noindent
{\bf Density self-correction.} It is not hard to see that when $\zeta(t,x)$ is small (much less than one), and this also holds in a region around $x$, then $\zeta(t,\cdot)$ will exhibit exponential growth near $x$, at least for a short time. 
Indeed, we heuristically have 
\[
\frac{\mathrm{d}}{\mathrm{d} t} \zeta(t,x) \approx \zeta(t,x) - \sum_{\{i: |X_i(t)-x| \in (0,1)\}} M_i(t) \cdot \zeta(t,X_i(t))
\, .
\]
This is not exactly correct since it ignores the effect of motion (particles may enter or leave the region near $x$), but it is a useful first approximation. In particular, it suggests that if $\zeta(t,y)$ is small (much less than one) for all $y$ with $|y-x| < 1$, then $\zeta(t,\cdot)$ will exhibit exponential growth near $x$, at least for a short time. This is indeed true; one important consequence is that if $\zeta(t,x) = \eps$ and $\zeta(t,\cdot)$ is not too wild then it is very likely that $\zeta(t',x) = \Theta(1)$ for some $t' = t+\Theta(\log(1/\eps))$. Similarly, when $\zeta(t,y)$ is much larger than $1$ for $y$ near $x$ then $\zeta(t,x)$ will decrease exponentially quickly. We use the self-correcting nature of the density in several places throughout the paper.

As an aside, we remark that if $\zeta(t,y) \approx \zeta(t,x)$ for $|y-x| < 1$ then the above heuristic gives 
$\frac{\mathrm{d}}{\mathrm{d} t} \zeta(t,x) \approx \zeta(t,x) (1-\zeta(t,x))$, which is suggestive of the logistic control; we briefly revisit this connection in the conclusion. 

\medskip \noindent
{\bf Population + no competitors=mass.} Fix $\beta > 0$ and suppose that for some function $f:[0,\infty) \to \R$, for all $s \in [0,t]$, $D(s,\beta) \le f(s)$, or in other words $\zeta(s,x) \le \beta$ for all $x \ge f(s)$. In this case, particles that stay ahead of the moving barrier $f$ are in a relatively sparse environment, so do not lose mass too quickly. 
More precisely, if $X_i(t)$ satisfies $X_{i,t}(s) \ge f(s)$ for all $s \in [0,t]$ then $M_i(t) \ge e^{-\beta t}$. 
It follows that for any $x \ge f(t)+1$, 
\[
\zeta(t,x) \ge e^{-\beta t} \cdot \# \{i: |X_i(t) - x| < 1, \forall s \in [0,t], X_{i,t}(s)> f(s)\}\, .
\]
For such $x$, if $\# \{i: |X_i(t) - x| < 1, \forall s \in [0,t], X_{i,t}(s)> f(s)\} > \beta e^{\beta t}$ then $\zeta(t,x) \ge \beta$, contradicting the assumption that $D(t,\beta) \le f(t)$. 

\medskip \noindent
{\bf Surfing the wave.} 
To exploit the above contradiction, we require that with high probability there are many particles staying ahead of some barrier. Such results are available: it  follows fairly straightforwardly from recent studies of {\em consistent maximal displacement} for BBM \cite{matt_cmd} that for $c>c^*$, for all large times $t$ there are $e^{\Theta(t^{1/3})}$ particles at time $t$ with which have stayed ahead of the curve $f(s)=\sqrt{2}s - cs^{1/3}$. This allows us to take $\beta=t^{-1}$ above and obtain that there is $s \in [0,t]$ and $x \ge f(s)$ such that $\zeta(s,x) \ge t^{-1}$. Since the local density grows exponentially in regions with small density, we will with high probability find $s'$ with $\zeta(s',x) > b > 0$ and $s'-s=O(\log t)$. For such $s'$ we have $f(s')=f(s) +O(\log  t)$ so $x \ge \sqrt{2}s' - c(s')^{1/3} - O(\log t)$ so with high probability $d(s',b) \ge \sqrt{2}s' - c(s')^{1/3} - O(\log t)$. 

The lower bound is practically complete, but we must rule out the possibility that $s'=O(1)$ for all $t$. To do so, we first establish that 
\[
\sup_{t >0} \frac{\max\{\zeta(t,x),x \in \R\}}{\log(t+2)}  =: Z <\infty \quad\mbox{almost surely}\, .
\]
Proving this is harder than might be expected; its proof, given in Section~\ref{sec:max_density}, occupies 8 pages and is perhaps the most technically challenging part of the paper. 

Once we prove that $Z < \infty$, we then reprise the above argument, but with a variable mass bound 
\[
\beta=\beta(s)=\begin{cases}
				Z\log(s+2)& \mbox{ for } s \le t^{1/4} \\
				t^{-1}	& \mbox{ for } s \in (t^{1/4},t]\, .
			\end{cases}
\]
The loss of mass before time $t^{1/4}$ is insignificant compared with that which follows, so essentially the same argument as above yields that there is $s \in [0,t]$ and $x \ge f(s)$ such that $\zeta(s,x) \ge \beta(s)$. On the other hand, this can not happen for $s < t^{1/4}$ by the definition of $Z$, so it must happen later. This is enough to conclude the lower bound. The details of this argument appear in Section~\ref{sec:lower}. 

\medskip \noindent
{\bf Competition implies decay.} 
For the upper bound, given in Section~\ref{sec:ub} (with some technical lemmas deferred to an appendix), we invert the above argument by contradiction. 
In brief: if all particles to the right of a given curve have spent large amounts of time in high-mass environments, then all such individuals will have very low weight; if furthermore there are not many of them, then their total weight is also small. 

More precisely, suppose that for some $C > 0$ and some function $g:[0,\infty) \to \R$, for all $s \in [Ct^{1/3}/2,t]$, we have $d(s,m) \ge g(s)$, so $\zeta(s,x) \ge m$ for all $x$ with $x\in [0,g(s)]$. 
Then for all $i$, 
\[
M_i(t) \le \exp(-m\cdot \mathrm{Leb}(s \in [Ct^{1/3}/2,t]: |X_{i,t}(s)| \in [0,g(s)]))\, .
\]
It follows that if all particles with $X_i(t) \ge g(t)$ have $\mathrm{Leb}(s \in [Ct^{1/3}/2,t]: |X_{i,t}(s)| \in [0,g(s)]) \ge \ell$ then 
for all $x \ge g(t)+1$, recalling the notation $N(t,x)$ from the introduction, 
\[
\zeta(t,x) \le e^{-m \ell} \cdot |N(t,g(t))|\, .
\]
If $|N(t,g(t))| \le m e^{m \ell}$, this is in contradiction with the assumption that $d(s,m) \ge g(s)$. 

\medskip \noindent
{\bf Whitecaps are just foam.} 
Once again using estimates related to consistent maximal displacement for BBM, we show that for $c<c^*$, with $g(s)=\sqrt{2}s-c s^{1/3}$, with high probability every particle with $X_i(t)> g(t)$ indeed spends at least a time $Ct^{1/3}$ behind the curve $g$. Under the above assumption, it follows that the particles counted by $N(t,g(t))$ are as insubstantial as sea spray; for all $x \ge g(t)+1$, 
\[
\zeta(t,x) \le e^{-mCt^{1/3}/2}\cdot \# \{i: |X_i(t) - x| < 1\} \le e^{-mCt^{1/3}/2} \cdot |N(t,g(t))|. 
\] 
Standard and simple arguments for BBM show that $|N(t,g(t))| = e^{O(t^{1/3})}$ with high probability, so we obtain a contradiction for large $t$ if $C$ is sufficiently large. It follows that with high probability there is $s \in [Ct^{1/3}/2,t]$ and $x\in [0,g(s)]$ such that $\zeta(s,x) <m$; for such $s$ we have $d(s,m) \le g(s)$. This is the content of Proposition~\ref{surf_prop}. 

\section*{Definitions}
We sometimes need to consider the evolution of a subset of the particles starting at a time greater than zero, so it is useful to allow initial conditions other than a single mass-one particle at the origin. Generally, for $\bx=(x_1,\ldots,x_k) \in \R^k$ and $\bm=(m_1,\ldots,m_k) \in (0,\infty)^k$, we write $\mathbf{P}_{\bx,\bm}$ for the probability measure corresponding to an initial condition 
with a particle of mass $m_i$ at location $x_i$ for each $1 \le i \le k$. We write $\mathbf{P}=\mathbf{P}_{(0),(1)}$ for the default initial condition. 

We say a random variable $X$ is {\em geometric with parameter $p$}, or is $\mathrm{Geom}(p)$-distributed, if $\p{X=k}=(1-p)^{k-1} p$ for positive integer $p$. 

\section{Upper bound}\label{sec:ub}

Recall from the introduction that $c^* = 3^{4/3}\pi ^{2/3}2^{-7/6}$. The next proposition is a restatement of the upper bound from Theorem~\ref{thm:main}. 

\begin{prop} \label{upper_bound}
For any $m>0$, almost surely
\[
\limsup _{t\rightarrow \infty} \frac{\sqrt 2 t-d(t,m)}{t^{1/3}}\geq c^*. 
\]
\end{prop}

For the remainder of the section, we fix $c \in (0,c^*)$ and let $g(s)=\sqrt 2 s-c s^{1/3}$ for $s\geq 0$. 
The following is the key step of the proof.

\begin{prop} \label{surf_prop} (``No one can surf $g$")
For any $C>0$, there exists $\delta = \delta (c,C)>0$ such that for $t$ sufficiently large  
\[
\p{\exists i \in N(t)\text{ s.t. } \mathrm{Leb} (\{s\leq t:X_{i,t}(s)\leq g(s)\})\leq Ct^{1/3}}\leq e^{-\delta t^{1/3}}. 
\]
\end{prop}
The proof of Proposition \ref{surf_prop} will take up most of this section, but we now give a brief justification of the result, and then show how it is used to prove Proposition \ref{upper_bound}. 
We shall choose a small constant $\beta>0$ and let $b(s)=\sqrt 2 s -c(s+\beta t)^{1/3}$ for $s\in [0,t]$.
Then by adapting the method used in \cite{jaffuel2012} for studying branching random walks, one may show that since $c<c^*$ if $\beta$ is sufficiently small then for any constant $K$,
\[
\p{\exists i\in N(t)\text{ s.t. }X_{i,t}(s)\geq b(s)-Kt^{1/6}\,\forall s\leq t}\leq e^{-\delta t^{1/3}} 
\]
for some $\delta>0$. 

Now fix $K>0$ large. For large $t$, the function $b$ is approximately linear on intervals of length $Ct^{1/3}$. This will allow us to use Brownian scaling to show that if $i\in N(t)$ only spends time $Ct^{1/3}$ time below $b$, then it has conditional probability at least $1/2$ of staying above $b-Kt^{1/6}$, so the probability such an $i\in N(t)$ exists is also $O(e^{-\delta t^{1/3}})$. 
Since $b\leq g$ this gives us Proposition \ref{surf_prop}.

Before giving the details of this argument, we prove Proposition \ref{upper_bound} assuming Proposition \ref{surf_prop}.
\begin{proof}[Proof of Proposition \ref{upper_bound}]
We continue to write $g(s)=\sqrt 2 s -cs^{1/3}$, for fixed $c \in (0,c^*)$ as above. Fix $m > 0$, let $C=4\sqrt 2 c(1+m^{-1})$, and let $\delta = \delta (c, C)$ be as defined in Proposition \ref{surf_prop}. It suffices to show that, as $t\rightarrow \infty$,
\begin{equation} \label{upper_bound_aim}
\p{\exists s\in [Ct^{1/3}/2,t]:d(s,m)\le g(s)+1}\rightarrow 1.
\end{equation}

Next, fix $t$ large. Recalling the notation $N(t,x)=\{i\in N(t):X_i(t)\geq x\}$, let
\[
A_1 =\{i\in N(t,g(t)):\mathrm{Leb} (\{s\leq t:X_{i,t}(s)\le g(s)\})\leq Ct^{1/3}\} 
\]
\[
\text{and }A_2 =\{i\in N(t,g(t)):\exists s\in[Ct^{1/3}/2,t]:X_{i,t}(s)<0\}.  
\]
Also, let $E$ be the event that $d(s,m)>g(s)+1$ for all $s\in [Ct^{1/3}/2,t)$. On the event $E$, if $i\in N(t,g(t))$ and $i\notin A_1 \cup A_2 $ then 
\[
M_i (t)\leq \exp(-m \mathrm{Leb} (\{Ct^{1/3}/2\le s\leq t:X_{i,t}(s)\in (0, g(s))\}))\leq \exp(-mCt^{1/3}/2). 
\]
Since all masses are at most $1$, it follows that on $E$,
\[
\sum_{i\in N(t,g(t))}M_i(t)\leq |A_1 \cup A_2|+ \exp(-mCt^{1/3}/2) |N(t,g(t))|.
\]
Also, for all $y\geq g(t)+1$ we have $\zeta (t,y)\leq \sum_{i\in N(t,g(t))}M_i(t)$; we thus have 
\begin{align}
& \,  \p{\forall s\in [Ct^{1/3}/2,t],d(s,m)> g(s)+1} \nonumber\\
= & \, 
\p{d(t,m)> g(t)+1,E} \nonumber\\
\le & \,  \p{\sum_{i\in N(t,g(t))}M_i(t)>m,E}\nonumber\\
\leq& \,  \p{A_1\cup A_2\neq \emptyset}+\p{|N(t,g(t))|\geq m \exp(mCt^{1/3}/2)}\, . \label{eq:threebounds}
\end{align}

By Proposition \ref{surf_prop}, for $t$ sufficiently large, $\p{A_1\neq \emptyset}\leq \exp (-\delta t^{1/3}).$ Next, using a spinal change of measure,
\[
\p{A_2 \neq \emptyset}\leq \E {|A_2|}\leq e^{\sqrt 2 c t^{1/3}}\p{B(t) \geq -ct^{1/3},\exists s\in [Ct^{1/3}/2,t]:B(s) \leq -\sqrt 2 s}. 
\]
Now partitioning according to the first interval $[j,j+1]$ in which $B(s) \leq - \sqrt 2 s$,
\begin{align*}
\p{\exists s\in [Ct^{1/3}/2,t]:B(s) \leq -\sqrt 2 s}&\leq \sum_{j=\lfloor Ct^{1/3}/2\rfloor}^{\lfloor t \rfloor}\p{\sup_{s\in [j,j+1]}B(s) \geq \sqrt 2 j}\\
&\leq \sum_{j=\lfloor Ct^{1/3}/2\rfloor}^{\lfloor t \rfloor}\p{\sup_{s\le j+1}B(s) \geq \sqrt 2 j}\\
&= 2\sum_{j=\lfloor Ct^{1/3}/2\rfloor}^{\lfloor t \rfloor}\p{B(j+1) \geq \sqrt 2 j}\\
&\leq 2 \sum_{j=\lfloor Ct^{1/3}/2\rfloor}^{\lfloor t \rfloor}\exp(-j^2/(j+1))\\
&\leq 2 \exp(-Ct^{1/3}/3)
\end{align*}
where the equality in the third line follows by the reflection principle, and the final inequality holds for $t$ sufficiently large. Since $C>4\sqrt 2 c$, it follows that 
\[
\p{A_2 \neq \emptyset}\leq e^{\sqrt 2 c t^{1/3}} 2 e^{-Ct^{1/3}/3}\le 2e^{-\sqrt 2 c t^{1/3}/3}. 
\] 
Finally, by another spinal change of measure,
\[
\p{|N(t,g(t))|>x}\leq x^{-1}\E{|N(t,g(t))|}=x^{-1}\E{e^{-\sqrt 2 B(t)}\I{B(t){ \geq -ct^{1/3}}}}\leq x^{-1}e^{\sqrt 2 c t^{1/3}}. 
\]
Combining the bounds on $\p{A_1 \ne \emptyset}$, $\p{A_2 \ne \emptyset}$, and $\p{|N(t,g(t))|>x}$ with 
(\ref{eq:threebounds}), we obtain that 
\[
\p{\forall s\in [Ct^{1/3}/2,t],d(s,m)> g(s)+1} \le 
e^{-\delta t^{1/3}}+2e^{-\sqrt 2 c t^{1/3}/3}+m^{-1}e^{(\sqrt 2 c -mC/2)t^{1/3}}\, ,
\]
which tends to $0$ as $t\rightarrow \infty$ since $C>2\sqrt 2 c m^{-1}$.
This establishes \eqref{upper_bound_aim} and completes the proof.
\end{proof}

For the rest of this section we work towards the proof of Proposition \ref{surf_prop}.
We shall need the following lemma.
Recall that we fixed $c<c^*$.
\begin{lem} \label{tailor_made_roberts}
There exists $\beta>0$ such that for $b(s)=\sqrt 2 s -c(s+\beta t)^{1/3}$ and for $K>0$, $t>0$ both sufficiently large, there exists a function $\Delta:[0,t]\to [t^{1/4},Kt^{1/3}]$ with $\Delta (t)\leq Kt^{1/4}$ and with $|\Delta ' (s)|\le 1$ for all $s\in [0,t]$, such that for all $u\in [0,t]$ and all $x\in [-Kt^{1/6},\Delta (u))$,
\begin{align}
&\p{b(s)-Kt^{1/6}<B(s)<b(s)+\Delta(s) \, \forall s\leq u, B(u)>b(u)+x}\nonumber\\
&\hspace{2cm} \leq \exp (-u-t^{1/3}/K +\sqrt 2 (\Delta (u)-x)).\label{eq:tmr}
\end{align}
\end{lem}
We prove Lemma~\ref{tailor_made_roberts} by appealing to technical lemmas from \citep{matt_cmd}, which bound the probability that a Brownian motion stays in a narrow tube of variable width. In order to verify that the results of \citep{matt_cmd} apply for some function $\Delta$ with the above properties, we adapt a technique from \cite{jaffuel2012}. In \cite{jaffuel2012}, the existence of a function analogous to $\Delta$ is constructed as the solution of a certain integral equation. We defer the details of the proof to Appendix \ref{append_upper}. 

From this point on, we let $\beta> 0$ and $b(s)$ be as in Lemma \ref{tailor_made_roberts}. We assume that $t$ is sufficiently large that $b$ is increasing on $[0,\infty)$. 
We now show that if $K$ is sufficiently large, a Brownian motion which spends at most $C t^{1/3}$ time before time $t$ below the curve $b$ has a conditional probability of at least $1/2$ of staying above the curve $b-Kt^{1/6}$ up to time $t$.

\begin{lem} \label{tech_lemma}
Let $(B(s))_{s\geq 0}$ be a Brownian motion started at $0$. Then given $C>0$, there is a constant $K(C)>0$ such that for $t$ sufficiently large, and any measurable function $\Delta:[0,t]\rightarrow (0,\infty)$, 
\begin{align*}
&\p{B(s) \leq b(s)+\Delta(s)\, \forall \, s\leq t, \mathrm{Leb} (\{s\leq t:B(s)\leq b(s)\})\leq Ct^{1/3}}\\
&\le 2 \p{b(s)-Kt^{1/6} < B(s) < b(s)+\Delta(s)\, \forall \, s\leq t}.
\end{align*}
\end{lem}
In proving Lemma~\ref{tech_lemma}, we will use the following auxiliary result.
\begin{lem}\label{lem:gb}
Fix non-negative real numbers $(t_i,i \ge 1)$
For each $i \ge 1$ let $(X_i(u),0 \le u \le t_i)$ be either a Brownian meander or a Brownian excursion of length $t_i$. Then writing $T=\sum_{i\ge 1} t_i$, for $x \ge 8 T^{1/2}$ we have 
\[
\p{\max_{i \ge 1} \max_{u \le t_i} X_i(u) \ge x}< e^{-x^2/16T}. 
\]
\end{lem}
The proof of Lemma~\ref{lem:gb} is deferred to the appendix. 
\begin{proof} [Proof of Lemma \ref{tech_lemma}]
Write
\begin{align*}
E & = \{B(s) \leq b(s)+\Delta(s)\, \forall \, s\leq t\}\\
A_1 & =\{B(s) \ge b(s)-Kt^{1/6}\, \forall \, s\leq t\}\, ,\\
A_2 &= \{\mathrm{Leb} (\{s\leq t:B(s)\leq b(s)\})\leq Ct^{1/3}\}\, .
\end{align*}
To prove the lemma, it suffices to show that provided $K=K(C)$ is sufficiently large, $\probC{A_1^c}{A_2 \cap E} \le 1/2$, since 
\[
\p{A_1 \cap E} \ge \p{A_1\cap A_2 \cap E}  = \p{A_2 \cap E}(1-\Cprob{A_1^c}{A_2 \cap E}).
\]

Fix $L \in (Ct^{1/3},2Ct^{1/3}]$ so that $n:=t/L$ is integer; this is possible for $t$ large enough. Then, for each $i \le n-2$ let $b_i:[iL,(i+2)L] \to \R$ be defined by 
\begin{equation}\label{eq:to_justify}
b_i(s) = b(iL) + \frac{s-iL}{2L} (b((i+2)L)-b(iL)) - 1.
\end{equation}
Note that by convexity and since the linear terms cancel, for all $i \le n-2$ and $s \in [iL,(i+2)L]$, 
\begin{align*}
b(iL)+ \frac{s-iL}{2L} (b((i+2)L)-b(iL)) & \le b(s)+
2L |b'(i+2L)-b'(iL)| \\
& \le b(s) + (2L)^2 \cdot \frac{2c}{9(\beta t+iL)^{5/3}} \\
& \le b(s) +\frac{32cC^2}{9 \beta^{5/3} t}\, ,
\end{align*}
which is less than $b(s)+1$ for $t$ sufficiently large. 
It follows that for $t$ sufficiently large, $b_i \le b$ on the interval $[iL,(i+2)L]$, for all $i \le n-2$.

Next, for $i \le n-2$ let $g_i = \inf\{s \ge iL: B(s) \ge b_i(s)\}$. Also, for $i < n-2$ let $d_i =\sup\{s \le (i+2)L: B(s) \ge b_i(s)\}$, and let $d_{n-2}=t$. Then write 
\[
\mathcal{U}_i = \{s \in [g_i,d_i]: B(s) \le b_i(s)\}.
\]
For $i < n-2$, this is the set of times when $B$ is performing an excursion below $b_i$ which starts at or after time $iL$ and ends at or before time $(i+2)L$. 
For $i=n-2$ we have $(i+2)L=t$, and in this case we include a final excursion below $b_i$ which does not end before time $t$ if it starts at or after time $iL$. 
The set $\mathcal{U}_i$ is a union of closed intervals, which we enumerate as $\{[l_{i,j},r_{i,j}],j \ge 1\}$ according to a fixed rule (in decreasing order of size, say). 

For all $i < n-2$, conditional on $\mathcal{U}_i$, for each $j \ge 1$ the function 
\begin{equation}\label{eq:b_ex}
(b_i(l_{i,j}+s)-B(l_{i,j}+s),0 \le s \le r_{i,j}-l_{i,j}) 
\end{equation}
is a Brownian excursion of length $r_{i,j}-l_{i,j}$.
The case $i=n-2$ is very slightly different, and we now describe it; for the remainder of the paragraph set $i=n-2$. If $B(t) \ge b_{n-2}(t)$ then there is no change. 
However, if $B(t) < b_{n-2}(t)$ then there there is a unique integer $j \ge 1$ with $[l_{i,j},r_{i,j}]$ with $r_{i,j}=t$; for this $j$ the process described by (\ref{eq:b_ex}) is a Brownian meander of length $r_{i,j}-l_{i,j}$; for all other $j$ the process is a Brownian excursion. All this is true even if we additionally condition on $A_2 \cap E$, since letting $\mathcal U=\cup_{i\le n-2}\mathcal{U}_i$, the occurrence of the event $A_2 \cap E$ is determined by $\mathrm{Leb}(\mathcal U)$ and $B|_{[0,t]\setminus \mathcal U}$.
By Lemma~\ref{lem:gb}, it follows that 
\begin{align}
& \Cprob{\sup_{s \in \mathcal{U}_i} (b_i(s)-B(s)) \ge x}{\mathcal{U}_i,A_2 \cap E} \nonumber\\
= & \Cprob{\sup_{j \ge 1}\sup_{s \in [l_{i,j},r_{i,j}]} (b_i(s)-B(s)) \ge x}{\mathcal{U}_i,A_2 \cap E} \nonumber\\
\le & 
\exp\left(-\frac{x^2}{16\mathrm{Leb}(\mathcal{U}_i)}\right) + \I{x^2 < 64 \mathrm{Leb}(\mathcal{U}_i)}\, .\label{eq:max_bd}
\end{align}
We next analyze the event $A_1^c \cap A_2$. 
Note that $b_i+1$ is the linear interpolation of $b$ on the interval $[iL,(i+2)L]$. Since $b$ is convex it follows that $b \le b_i+1$ on this interval. 

If $A_1^c$ occurs then 
there is $s \le t$ such that $B(s) \le b(s) - K t^{1/6}$. 
For such $s$, for any $i$ with $s \in [iL,(i+2)L]$, 
the preceding paragraph then implies that $B(s) \le b(s) -Kt^{1/6} \le b_i(s) - (Kt^{1/6}-1)$. 

Next suppose $A_2$ occurs, and suppose $s\le t$ is such that $B(s) \le b(s)-Kt^{1/6}$. Then $s$ is in an excursion of $B(s)$ below $b(s)$. Temporarily write $[g,d]$ for the time interval during which this excursion takes place. Since $A_2$ occurs, $[g,\min(d,t)]$ has length at most $Ct^{1/3}$ so is strictly contained within in an interval $[iL,(i+2)L]$ for some $i \le n-2$. Since $s \in [g,\min(d,t)]$ and $B(g) = b(g) \ge b_i(g)$ and either $ d \ge t$ or $B(d) = b(d) \ge b_i(d)$ but $B(s) < b_i(s)$, it follows that $s \in \mathcal{U}_i$. 
On the other hand, each point $s$ lies in at most three distinct sets $\mathcal{U}_i$, so on $A_2$ we have 
\[
\sum_{i \le n-2} \mathrm{Leb}(\mathcal{U}_i) \le 3Ct^{1/3}\, .
\]

Finally, suppose $A_1^c \cap A_2$ occurs. 
Then the observations of the preceding three paragraphs imply that there exists $i \le n-2$ and 
$s \in \mathcal{U}_i$ such that $B(s) \le b_i(s) - (Kt^{1/6}-1) < b_i(s)- Kt^{1/6}/2$, the last inequality holding for $t$ large. 
Combined with (\ref{eq:max_bd}), this yields 
\begin{align*}
& \Cprob{A_1^c}{A_2 \cap E} \\
 \le&
 \Cprob{\sup_{i \le n-2}\sup_{s \in \mathcal{U}_i}
(b_i(s)-B(s)) \ge Kt^{1/6}/2,\sum_{i \le n-2} \mathrm{Leb}(\mathcal{U}_i) \le 3C t^{1/3}}{A_2 \cap E}\, \\
\le&  
\mathop{\sup_{u_1+\ldots+u_{n-2} \le 3Ct^{1/3}}}_{u_i \ge 0}
\sum_{i=1}^{n-2} 
 \Cprob{\sup_{s \in \mathcal{U}_i} (b_i(s)-B(s)) \ge Kt^{1/6}/2}{\mathrm{\mathrm{Leb}}(\mathcal{U}_i)=u_i,A_2\cap E} \\
 \le&\mathop{\sup_{u_1+\ldots+u_{n-2} \le 3Ct^{1/3}}}_{u_i \ge 0} \sum_{i=1}^{n-2} \exp(-K^2t^{1/3}/64u_i)\, ,
\end{align*}
the last bound holding provided that $K^2> 768C$ so that $(Kt^{1/6}/2)^2 > 64(3Ct^{1/3})$. 
Finally, letting $x=K^2t^{1/3}/64$, the function $f(a)=e^{-x/a}$ is convex for $a\in [0,x/2]$, and $f(0)=0$, so if  $K^2> 384C$ then for each $i$, $f(u_i)\leq (u_i/\sum u_k) f(\sum u_k)$. Hence 
\[
 \Cprob{A_1^c}{A_2 \cap E}  \le e^{-K^2/192C} < e^{-4}<1/2, 
\]
as required. 
\end{proof}

We next state a variant of Lemma \ref{tech_lemma}  which is proved in a similar way.

\begin{lem} \label{tech_lemma_variant}
Let $(B(s))_{s\geq 0}$ be a Brownian motion started at $0$. Then given $C>0$, there is a constant $K=K(C)$ such that for $t$ sufficiently large, and any measurable function $\Delta:[0,t]\rightarrow (0,\infty)$, $u\le t$ and $z\in [b(u), b(u)+\Delta(u))$, we have
\begin{align*}
&\p{B(s) \leq b(s)+\Delta(s)\, \forall \, s\leq u, \mathrm{Leb} (\{s\leq u:B(s)\leq b(s)\})\leq Ct^{1/3}, B(u)\geq z}\\
&\le 2 \p{b(s)-Kt^{1/6} < B(s) < b(s)+\Delta(s)\, \forall \, s\leq u, B(u)\geq z}.
\end{align*}
and
\begin{align*}
&\p{B(s) \leq b(s)+\Delta(s)\, \forall \, s\leq u, \mathrm{Leb} (\{s\leq u:B(s)\leq b(s)\})\leq Ct^{1/3}}\\
&\le 2 \p{b(s)-Kt^{1/6} < B(s) < b(s)+\Delta(s)\, \forall \, s\leq u}.
\end{align*}
\end{lem}
\begin{proof}
These bounds are proved in the same way as Lemma \ref{tech_lemma}, by only considering the times $(\mathcal U_i)_i$ when $B$ is performing an excursion below $b_i$ on the interval $[0,u]$, and using that $z\ge b(u)$, conditioning on $B(u)\ge z$ does not affect the distribution of $B$ on $(\mathcal U_i)_i$ given $(\mathcal U_i)_i$. We omit the details. 
\end{proof}
We are now in a position to complete the proof of Proposition \ref{surf_prop}, using Lemmas \ref{tailor_made_roberts}, \ref{tech_lemma} and \ref{tech_lemma_variant}.
\begin{proof}[Proof of Proposition \ref{surf_prop}]
Choose $\beta$ such that that Lemma \ref{tailor_made_roberts} applies, and recall that $b(s)=\sqrt 2 s - c (s+\beta t)^{1/3}\leq g(s)$ $\forall s\geq 0$. Then 
\begin{align*}&\p{\exists i \in N(t)\text{ s.t. } \text{Leb} (\{s\leq t:X_{i,t}(s)\leq g(s)\})\leq Ct^{1/3}}\\
&\hspace{2cm}\leq \p{\exists i \in N(t)\text{ s.t. } \text{Leb} (\{s\leq t:X_{i,t}(s)\leq b(s)\})\leq Ct^{1/3}}.
\end{align*}
We shall prove that 
\begin{equation}  \label{want_to_show}
\p{\exists i \in N(t)\text{ s.t. } \text{Leb} (\{s\leq t:X_{i,t}(s)\leq b(s)\})\leq Ct^{1/3}}\leq e^{-\delta t^{1/3}}
\end{equation}
for some $\delta >0$ for $t$ sufficiently large, which establishes the proposition. 
Take $K$ and $t$ sufficiently large that Lemmas~\ref{tailor_made_roberts},~\ref{tech_lemma} and~\ref{tech_lemma_variant} hold. 
Then let $\Delta:[0,t]\to [t^{1/4},Kt^{1/3}]$ be as in Lemma~\ref{tailor_made_roberts}, and in particular satisfying that  $\Delta (t)\leq Kt^{1/4}$ and $|\Delta ' (s)|\le 1$ for all $s\in [0,t]$.

Since $|\Delta ' (s)|\le 1$ for all $s\in [0,t]$, $\inf_{u\in [j,j+1]} \Delta (u) \ge \Delta (j)-1$ for $j\in [0,t-1]$. Hence if for some $i\in N(j+1)$, $X_{i,j+1}(s)\ge b(s)+\Delta(s)$ for some $s\in [j,j+1]$, then since $b$ is increasing,
\begin{equation} \label{partition_delta_bound}
X_{i,j+1}(s)\ge b(j)+\inf_{u\in [j,j+1]} \Delta (u)\ge b(j)+ \Delta (j)-1.
\end{equation}
Using \eqref{partition_delta_bound}, and partitioning the event $\{\exists i\in N(t)\text{ s.t. Leb}(\{s\leq t:X_{i,t}(s)\leq b(s)\})\leq Ct^{1/3}\}$ according to the interval $[j,j+1]$ in which $X_{i,t}(s)$ first exceeds $b+\Delta$, we have that
\begin{align*}
&\p{\exists i\in N(t)\text{ s.t. Leb}(\{s\leq t:X_{i,t}(s)\leq b(s)\})\leq Ct^{1/3}}\\
&\hspace{0.5cm}\leq
\p{\exists i \in N(t)\text{ s.t. }X_{i,t}(s)\leq b(s)+\Delta(s)\, \forall \, s\leq t, \text{Leb}(\{s\leq t:X_{i,t}(s)\leq b(s)\})\leq Ct^{1/3}}\\
&\hspace{1cm}+\sum_{j=0}^{\lfloor t \rfloor} \mathbf P \{\exists i \in N(j+1)\text{ s.t. }X_{i,j+1}(s)\leq b(s)+\Delta(s)\, \forall \, s\leq j,\\
&\hspace{2.5cm} \text{Leb}(\{s\leq j:X_{i,j+1}(s)\leq b(s)\})\leq Ct^{1/3},\sup_{s\in [j,j+1]}X_{i,j+1}(s)\geq b(j)+\Delta(j)-1\}\\
&\hspace{0.5cm}\leq e^t \p{B(s)\leq b(s)+\Delta(s)\, \forall \, s\leq t, \text{Leb}(\{s\leq t:B(s)\leq b(s)\})\leq Ct^{1/3}}\\
&\hspace{1cm}+\sum_{j=0}^{\lfloor t \rfloor}e^{j+1}\mathbf P \{B(s)\leq b(s)+\Delta(s)\, \forall \, s\leq j,\\
&\hspace{4cm} \text{Leb}(\{s\leq j:B(s)\leq b(s)\})\leq Ct^{1/3},\sup_{s\in [j,j+1]}B(s)\geq b(j)+\Delta(j)-1\},
\end{align*}
where the last inequality follows by Markov's inequality and the many-to-one lemma.
By partitioning according to the value of $B(j)$, we further have
\begin{align*}
&\p{B(s)\leq b(s)+\Delta(s)\, \forall \, s\leq j,\text{Leb}(\{s\leq j:B(s)\leq b(s)\})\leq Ct^{1/3},\sup_{s\in [j,j+1]}B(s)\geq b(j)+\Delta(j)-1}\\
&\leq \p{B(s)\leq b(s)+\Delta(s)\, \forall \, s\leq j,\text{Leb}(\{s\leq j:B(s)\leq b(s)\})\leq Ct^{1/3}, B(j)\geq b(j)+\Delta(j)-\tfrac{1}{2}t^{1/4}}\\
&\hspace{0.5cm}+\p{\sup_{[0,1]}B(u)\geq \tfrac{1}{2}t^{1/4}}\p{B(s)\leq b(s)+\Delta(s)\, \forall \, s\leq j,\text{Leb}(\{s\leq j:B(s)\leq b(s)\})\leq Ct^{1/3}}.
\end{align*}
Since $\Delta (j)-\tfrac{1}{2}t^{1/4}>0$, we can now apply Lemma \ref{tech_lemma} and Lemma \ref{tech_lemma_variant} to conclude that
\begin{align} \label{leb_under_f}
&\p{\exists i\in N(t)\text{ s.t. Leb}(\{s\leq t:X_{i,t}(s)\leq b(s)\})\leq Ct^{1/3}}\nonumber \\
&\hspace{1cm}\leq 2 e^t \p{b(s)-Kt^{1/6} < B(s) < b(s)+\Delta(s)\, \forall \, s\leq t}\nonumber\\
&\hspace{1.2cm}+2\sum_{j=0}^{\lfloor t \rfloor}e^{j+1}\bigg(\p{b(s)-Kt^{1/6} < B(s) < b(s)+\Delta(s)\, \forall \, s\leq j,B(j)\geq b(j)+\Delta(j)-\tfrac{1}{2}t^{1/4}}\nonumber\\
&\hspace{3.8cm}+ 2e^{-\tfrac{1}{8}t^{1/2}}\p{b(s)-Kt^{1/6} < B(s) < b(s)+\Delta(s)\, \forall \, s\leq j}\bigg).
\end{align}
We can now apply Lemma \ref{tailor_made_roberts} to each term. First, by Lemma \ref{tailor_made_roberts} applied with $u=t$ and $x=-Kt^{1/6}$, since $\Delta (t) \le Kt^{1/4}$,
\[
\p{b(s)-Kt^{1/6} < B(s) < b(s)+\Delta(s)\, \forall \, s\leq t}\le \exp (-t -t^{1/3}/K +\sqrt 2 K(t^{1/4}+t^{1/6})). 
\]
By Lemma \ref{tailor_made_roberts} applied with $u=j$ and $x=\Delta (j)-\tfrac{1}{2}t^{1/4}$,
\begin{align*}
&\p{b(s)-Kt^{1/6} < B(s) < b(s)+\Delta(s)\, \forall \, s\leq j,B(j)\geq b(j)+\Delta(j)-\tfrac{1}{2}t^{1/4}}\\
&\hspace{1cm}\le \exp(-j-t^{1/3}/K+\sqrt 2 \tfrac{1}{2} t^{1/4}).
\end{align*}
Finally, by Lemma \ref{tailor_made_roberts} applied with $u=j$ and $x=-Kt^{1/6}$, since $\Delta(j)\le Kt^{1/3}$,
\[
\p{b(s)-Kt^{1/6} < B(s) < b(s)+\Delta(s)\, \forall \, s\leq j}
\le \exp(-j-t^{1/3}/K+\sqrt 2 K (t^{1/3}+t^{1/6})). 
\]
Putting everything together in \eqref{leb_under_f},
\begin{align*}
&\p{\exists i\in N(t)\text{ s.t. Leb}(\{s\leq t:X_{i,t}(s)\leq b(s)\})\leq Ct^{1/3}}\\
&\hspace{1cm}\le 2\exp (-t^{1/3}/K +\sqrt 2 K(t^{1/4}+t^{1/6}))\\
&\hspace{1.2cm}+2e\sum_{j=0}^{\lfloor t \rfloor}\bigg(\exp(-t^{1/3}/K+\sqrt 2 \tfrac{1}{2} t^{1/4})+ 2\exp(-t^{1/2}/8+O(t^{1/3}))\bigg)\\
&\hspace{1cm}\leq e^{-\delta t^{1/3}}
\end{align*}
for some $\delta >0$ for $t$ sufficiently large, which proves \eqref{want_to_show}.
\end{proof}

\section{The greatest overall particle density}\label{sec:max_density}
Before moving to the lower bound, we first prove logarithmic upper bounds on how the greatest particle density grows over time; these are needed to ensure that particle masses cannot {\em decay} too quickly. This may seem contradictory, but the point is that a particle may {\em a priori} quickly lose a large amount of mass if it finds itself in an extremely dense environment. The next proposition rules this out. 
\begin{prop}
\label{prop:massbound}
Let $Z=2\cdot 10^8$; then for all $s$ sufficiently large, 
\[\p{\sup\{\zeta(t,x): 0 \le t \le s, x \in \R\} > Z\log s} \le s^{-4}.\]
\end{prop}
Proving Proposition~\ref{prop:massbound} turns out to be a fair amount of work. In order that the idea is not obscured by detail, however, we set up the heart of the argument right away. 

Let $z(t,x) = \sum_{\{i: |X_i(t)-x|< 1/2\}}M_i(t)$. The differences between $z$ and $\zeta$ are that $z$ only counts 
mass within distance $1/2$ of $x$, and does not ignore the mass of particles at $x$ (should there be any). 

Let $z(t) = \sup_x z(t,x)$, and define a sequence $(\tau_i,i \ge 0)$ of stopping times as follows. 
Fix $s$ large and for the remainder of the section write $N=N(s)=10^7\log s$. Let $\tau_0=\inf\{t:z(t) \ge N-1\}$, and for $k \ge 0$ let $\tau_{k+1} = \inf\{t > \tau_k+10^5/N: z(t) \ge N-1\}$. 
Then $\tau_k \ge 10^5k/N$, so with $I=I(s) = \inf\{k: \tau_k \ge s\}$, we have $I \le \lceil Ns/10^5 \rceil$ and 
\[
\sup\{z(t),t \le s\} \le \sup \{z(t),t < \tau_{I}\}. 
\]

Notice that the sequence of stopping times ``ignores'' small time intervals $[\tau_k,\tau_k+10^5/N]$. 
However in any time interval $[\tau_k+10^5/N,\tau_{k+1})$, the function $z$ nowhere exceeds $N$ by the definition of the stopping time $\tau_{k+1}$. 
We thus have 
\begin{equation}\label{eq:mass_keybound}
\sup\{z(t),t \le s\} \le \sup \{z(t),t < \tau_{I}\} \le 
\max\pran{N,\sup_{k < I} \sup_{t \in [\tau_k,\tau_k+10^5/N]} z(t)}
\end{equation}
We prove the proposition by establishing the following facts. The first fact says that for $k < Ns/10^5$, if $z(\tau_k)$ is not too large then with high probability $z(t)$ is not too large for any $t \in [\tau_k,\tau_k+10^5/N]$. The second says that for such $k$, with high probability $z(\tau_k+10^5/N)$ is small. 
\begin{fact}\label{fact:massdoub}
For $s$ sufficiently large, for all $0 \le k < Ns/10^5$, 
\[
\p{\sup\{z(t),t \in [\tau_k,\tau_k+10^5/N]\} > 10N,z(\tau_k)\le N,k < I} < s^{-6}. 
\]
\end{fact}
\begin{fact}\label{fact:masslow}
For $s$ sufficiently large, for all $0 \le k < Ns/10^5$, 
\[
\p{z(\tau_k+10^5/N) \ge N-1,k < I} < s^{-6}. 
\]
\end{fact}
Assuming these two facts, the proposition follows easily. 
\begin{proof}[Proof of Proposition~\ref{prop:massbound}]
Fix $k \le Ns/10^5$. 
Note that if $z(\tau_{k-1}+10^5/N) < N-1$ then $z(\tau_{k}^-) < N-1$. Since mass only increases by branching, it follows 
that almost surely a single branching event at time $\tau_{k}$ causes $z$ to increase above $N-1$. As all masses are at most $1$ and branching is binary, it follows that in this case almost surely $z(\tau_k) \le z(\tau_k^-)+1 < N$. 
With Fact~\ref{fact:masslow}, this implies that 
\begin{align*}
\p{z(\tau_{k}) > N,k < I}
& \le \p{z(\tau_{k-1}+10^5/N)\ge N-1,k<I}\\
& \le  \p{z(\tau_{k-1}+10^5/N)\ge N-1,k-1<I} \\
&< s^{-6}\, .
\end{align*}

We now use that for any events $A,B,C$ we have $\p{A\cap C} \le \p{A \cap B\cap C} + \p{B^c \cap C}$. 
By Fact~\ref{fact:massdoub} and the preceding bound, we obtain that for $0 \le k < Ns/10^5$, 
\[
\p{\sup\{z(t),t \in [\tau_k,\tau_k+10^5/N]\} > 10N, k < I} \le 2s^{-6}
\]
A union bound and (\ref{eq:mass_keybound}) then yield 
\begin{align*}
\p{\sup_{t \le s} z(t) > 10N} & 
\le 
\p{ \sup_{k <I} \sup_{t \in [\tau_k,\tau_k+10^5/N]} z(t) > 10N} \\
& \le \sum_{k=0}^{\lfloor Ns/10^5\rfloor} \p{\sup_{t \in [\tau_k,\tau_k+10^5/N]} z(t) > 10N, k < I} \\
& \le \pran{1+\frac{Ns}{10^5}} \cdot 2s^{-6}\, \\
& < s^{-4}\, ,
\end{align*}
the last inequality holding for $s$ large. 
Finally, it is easy to see that $\sup_x \zeta(t,x) \le 2 z(t)$, so the same bound holds for 
$\p{\sup_{t \le s} \sup_x \zeta(t,x) > 20N}$, which proves the proposition.\end{proof}

The reader who is willing to believe the Facts~\ref{fact:massdoub} and~\ref{fact:masslow} without proof -- or who is impatient to see how Proposition~\ref{prop:massbound} is used to prove the lower bound from the main theorem -- could skip directly to Section~\ref{sec:lower} at this point. 
 
\subsection{Proofs of Facts~\ref{fact:massdoub} and~\ref{fact:masslow}}\label{sec:twofacts}
We first prove a handful of technical estimates required for the proofs. 
The first shows that a fixed mass of particles is extremely unlikely to quickly increase its total mass. 
Recall the definition of $\mathbf{P}_{\bx,\bm}$ from just before the start of Section~\ref{sec:ub}. 
\begin{lem}\label{lem:decrease_2} 
Fix $\bx=(x_1,\ldots,x_k) \in \R^k$ and $\bm = (m_1,\ldots,m_k) \in [0,1]^k$. Under $\mathbf{P}_{\bx,\bm}$, for 
$1 \le j \le k$ let $G_j(s) = \# \{i: j_{i,s}(0)=j\}$ be the number of time-$s$ descendants of $x_j$. 
Then for any $J\subset \{1,\ldots, k\}$, any  $x\ge\sum_{j \in J} m_j$, for all $t \le \log 2$ and all $\delta > 0$, 
\[
\psub{\bx,\bm}{\sum_{j\in J} m_jG_j(t) \ge (1+\delta) x} \le 2 \pran{2^{1+\delta} (1-e^{-t})^{\delta}}^x.
\]
\end{lem}
\begin{proof}
We may clearly assume $J=\{1,\ldots,k\}$. 
Also, adding particles to increase the mass of the starting configuration can only increase the probability we aim to bound, so we may assume that $x = \sum_{i=1}^k m_i$. 
The random variables $(G_j(s),1 \le j \le k)$ are i.i.d.\  and are Geom$(e^{-s})$-distributed (see, e.g., \cite{mckean75}). 
Lemma~\ref{lem:geom_ld} provides upper tail bounds for weighted sums of geometric random variables where the individual coefficients are small compared with their sum. Using that lemma (with $\eps=1-e^{-t}$ -- this is where we require that 
$t < \log 2$), the result follows. 
\end{proof}
Since $G_j(s)$ is non-decreasing in $s$, we have
\[
\sup_{s \in [0,t]} \sum_{\{i:j_{i,s}(0)\in J\}} M_{i,s}(0) = \sup_{s \in [0,t]} \sum_{j\in J} m_j G_j(s) = \sum_{j \in J} m_j G_j(t)\, .
\]
Combining this with the preceding lemma thus also yields the following bound. 
\begin{cor}\label{cor:decrease_2}
With the hypotheses and notation of Lemma~\ref{lem:decrease_2}, 
\[
\psub{\bx,\bm}{\sup_{s \in [0,t]} \sum_{\{i:j_{i,s}(0)\in J\}} M_{i,s}(0) \ge (1+\delta) x} \le 2\pran{2^{1+\delta} (1-e^{-t})^{\delta}}^x.
\]
\end{cor}
The next proposition says that mass does not travel far in a short time, even once branching is taken into account. 
\begin{prop}\label{prop:mass_travel}
Fix $\bx=(x_1,\ldots,x_k) \in \R^k$ and $\bm = (m_1,\ldots,m_k) \in (0,1]^k$. Then for all  $x \ge  \sum_{1 \le i \le k} m_i$, 
for all $t > 0$, $L \ge 1/2$ and $v > 0$, we have 
\[
\psub{\bx,\bm}{\sum_{\{i:X_i(t)-X_{i,t}(0)> L\}} M_{i,t}(0) > vx} \le \frac{\exp(t-L^2/(2t))}{v}\, .
\]
\end{prop}
\begin{proof}
We begin with a few simplifying assumptions. First, we may clearly assume that $x_i=0$ for all $i\le k$. Next, adding particles to the system at time $0$ can only increase the probability we aim to bound, so we may assume that $x = \sum_{i=1}^k m_i$. 

For $j \le k$ write $S_j = \{i \le n(t): j_{i,t}(0)=j\}$  
for the set of indices of time-$t$ descendants of $x_j$. 
Then let 
$R_j = \{i \in S_j: X_i(t)-X_{i,t}(0) > L\}$, so that 
\[
\sum_{\{i:X_i(t)-X_{i,t}(0)> L\}} M_{i,t}(0) = 
\sum_{j=1}^k \sum_{i \in R_j} M_{i,t}(0) = 
\sum_{j=1}^k m_j |R_j|\, .
\]
By the many-to-one lemma, for $W$ a one-dimensional Brownian motion,
\[
\E{|R_j|}= e^t \p{W_t-W_0> L}\le \exp(t-L^2/(2t))
\]
for $L \ge 1/2$.
This bound does not depend on $j \le k$. It then follows by Markov's inequality that for $v > 0$, 
\begin{align*}
\p{\sum_{\{i:X_i(t)-X_{i,t}(0)> L\}} M_{i,t}(0) > vx} & = 
\p{\sum_{j \le k} m_j |R_j| > vx}\\
& \le 
\frac{\E{\sum_{j \le k} m_j |R_j|}}{vx}\\
& \le 
 \frac{\exp(t- L^2/(2t))}{v}\, ,
\end{align*}
where we have used in the last inequality that $\sum_{j \le k} m_j = x$. 
\end{proof}
In the sequel, we in fact use the following corollary, which extends Proposition~\ref{prop:mass_travel} by considering all times in an interval $[0,t]$, rather than a fixed time $t>0$, at the cost of a slightly weaker bound. 
\begin{cor}\label{cor:mass_travel}
Under the conditions of Proposition~\ref{prop:mass_travel}, for all $t_0 > 0$, $L\ge 1/2$ and $v > 0$,  and all $x \ge \sum_{i\le k} m_i$, 
\[
\psub{\bx,\bm}{\sup_{t \le t_0} \sum_{\{i:X_i(t)-X_{i,t}(0)\ge L\}} M_{i,t}(0) > 2vx} \le \frac{2 \exp(t_0-L^2/(2t_0))}{v} \, .
\]
\end{cor}
\begin{proof}
Consider the stopping time 
\[
\tau = \inf\left\{t: \sum_{\{i:X_i(t)-X_{i,t}(0)\ge L\}} M_{i,t}(0) > 2vx \right\}\, .
\]
By symmetry, 
\[
\Cprob{\sum_{\{i:X_i(t_0)-X_{i,t_0}(0)\ge L\}} M_{i,t_0}(0) > vx}{\tau \le t_0} \ge \frac{1}{2}\, ,
\]
and the corollary follows. 
\end{proof}

The next lemma says that a large, concentrated mass will quickly decay; once we prove this we will have all the tools we need to establish Facts~\ref{fact:massdoub} and~\ref{fact:masslow}. 
\begin{lem}\label{lem:decrease_1}
There exist $t_0>0$ and $C > 0$ such that the following holds. 
Fix $\bx=(x_1,\ldots,x_k) \in \R^k$ and $\bm = (m_1,\ldots,m_k) \in [0,1]^k$. 
Let $J = \{j: |x_j| < 1/4\}$, and suppose $A = \sum_{j \in J} m_j > C$. 
Then for all $t \in [500/A, t_0]$, 
setting $I = \{i: j_{i,t}(0) \in J\}$ we have 
\[
\mathbf{P}_{\bx,\bm}\left\{\sum_{i \in I} M_i(t) > A/24\right\} \le 2e^{-200A}\, .
\]
\end{lem}
\begin{proof}
The proof is divided as follows. First, the total mass at time $t$ of particles whose trajectory branches at least once is small. Next, among non-branching trajectories, the total mass which moves far from the origin is small. Finally, particles whose trajectories do not branch and stay near the origin will lose a large amount of mass since they are a dense environment. We now formalize this. 

Write $I_b = \{i \in I: \exists i' \ne i, j_{i,t}(0)=j_{i',t}(0)\}$ for the indices of particles starting near (distance $< 1/4$) to the origin whose trajectories branch before time $t$. 
Then let $I\setminus I_b= I_f \cup I_n$, where 
\[
I_f = \left\{i\in I\setminus I_b: |X_{i,t}(0)| < 1/4, \sup_{s \in [0,t]} |X_{i,t}(s)| > 1/2\right\}
\]
indexes non-branching trajectories that start near the origin but move far (distance $>1/2$) from the origin before time $t$, and 
where $I_n=I\setminus (I_f\cup I_b)$ indexes non-branching trajectories that stay near the origin. Then with $M_b= \sum_{i \in I_b} M_i(t)$ and $M_f$, $M_n$ defined accordingly, we have 
\[
\sum_{i \in I} M_i(t)=M_b + M_f+M_n. 
\]

We begin by considering branching trajectories. For each $1 \le j \le k$, let $G_j = \# \{i \in I: j_{i,t}(0)=j\}$. 
Then $i \in I_b$ precisely if $j_{i,t}(0) \in J$ and $G_{j_{i,t}(0)}> 1$. Since masses decrease with time, 
\begin{align*}
\sum_{i \in I_b} M_i(t)  	& \le \sum_{i \in I_b} m_{j_{i,t}(0)} \\	
					& = \sum_{j \in J} m_j G_j \I{G_j > 1}.
\end{align*}
Next, since the $G_j$ are integer-valued, 
\[
\sum_{j \in J} m_j G_j \I{G_j > 1} = \sum_{j \in J} m_j (G_j-1) + \sum_{j \in J} m_j \I{G_j > 1} < 2\sum_{j \in J} m_j (G_j-1), 
\]
which with the preceding bound gives 
\begin{align*}
\sum_{i \in I_b} M_i(t)	& \le 2 (\sum_{j \in J} m_j G_j - A). 
\end{align*}
By Lemma~\ref{lem:decrease_2}, it follows that for any fixed $\delta > 0$, if $t<\log 2$,
\begin{align}\label{eq:ib_bound}
\p{\sum_{i \in I_b} M_i(t) > 2\delta A} & \le 
\p{\sum_{j \in J} m_jG_j \ge (1+\delta) A} \nonumber\\ 
& \le 2(2^{1+\delta} (1-e^{-t})^{\delta})^A \nonumber\\
& < 2(2^{1+\delta} t^{\delta})^A \nonumber\\
& \le e^{-200A}\, ,
\end{align}
the last bound holding for $t$ sufficiently small that $2^{2+\delta} t^{\delta} < e^{-200}$. 
We next bound $\sum_{i \in I_n} M_i(t)$, the total final mass from ``typical'' trajectories, which do not branch and do not move far from their starting position by time $t$. 
Fix $c  \in (0,1)$ and let $E$ be the event that for all $s \in [0,t]$, $\sum_{\{i: |X_i(s)|<1/2\}} M_i(s) > cA$. On $E$, if $i \in I_n$ has $j_{i,t}(0)=j$ then $M_i(t) \le m_j \cdot e^{-tcA}$. 
We thus have 
\[
\sum_{i \in I_n} M_i(t) \I{E} \le \sum_{j \in J} m_j \cdot e^{-tcA}\cdot \I{E} = A e^{-tcA}\I{E}\, .
\]
Next, let $I_n(s)=\{j_{i,t}(s): i \in I_n\}$ be the indices of time-$s$ ancestors of individuals in $I_n$. 
Since trajectories indexed by $I_n$ do not branch, $\sum_{i \in I_n(s)} M_i(s)$ is decreasing for $s \in [0,t]$. 
Necessarily $|X_i(s)| < 1/2$ for $i \in I_n(s)$, so if $E^c$ occurs then there is $s \in [0,t]$ such that $\sum_{i \in I_n(s)} M_i(s) \le cA$. 
We thus have 
\[
\sum_{i \in I_n} M_i(t) \I{E^c} \le cA\I{E^c}\, ,
\]
and the two preceding bounds together give 
 \begin{align} \label{eq:in_bound}
\sum_{i \in I_n} M_i(t) & \le \max\pran{cA, A e^{-tcA}}\, .
 \end{align}

Finally, we turn to the final mass of non-branching trajectories that move far from the origin, counted by $\sum_{i \in I_f} M_i(t)$. 
For any $i \in I$, If $j_{i,t}(0)=j$ and $|x_j|< 1/4$ then in order to have $\sup_{s \in [0,t]} |X_{i,t}(s)| > 1/2$ the trajectory leading to $X_i(t)$ wanders a distance of at least $1/4$ from its starting position. 
Let $W$ denote one-dimensional Brownian motion started from the origin. By the reflection principle and the fact that $\p{G > x} \le e^{-x^2/2}/(2x)$ for $G$ a standard normal and for all $x > 0$, we have 
\[
\p{\sup_{s \le t} |W_s| > 1/4} \le 4\p{W_t > \frac{1}{4}} \le 8\exp(-1/(32t))\, . 
\]
Since an individual trajectory of $X$ has the law of Brownian motion, for a particle starting at distance less than $1/4$ from the origin whose trajectory never branched, the above is a bound on the probability the trajectory attained distance $1/2$ from the origin. It follows that 
\[ 
\sum_{i \in I_f} M_i(t) \pst \sum_{j \in J} m_j \cdot \xi_j\, ,
\]
where the terms $\xi_j$ are i.i.d.\  Ber$(8\exp(-1/(32t)))$. 
The variance of the latter sum is bounded by $A\cdot 8\exp(-1/(32t))$, so 
Theorem~\ref{thm:bern} yields that for any fixed $b > 0$, 
\begin{equation}\label{eq:if_bound}
\p{\sum_{i \in I_f} M_i(t) > (b+8\exp(-1/(32t)))A}  \le 
\pran{\frac{8e^{1-1/(32t)}}{b}}^{bA} 
< e^{-200A}\, ,
\end{equation}
the final inequality for $t$ sufficiently small. 

We now combine (\ref{eq:ib_bound}), (\ref{eq:in_bound}) and (\ref{eq:if_bound}). This yields that for $t$ sufficiently small, 
and in particular provided that $2^{2+\delta}t^{\delta}  < e^{-200}$, $((8e^{1-1/(32t)})/b)^{b}<e^{-200}$  and that 
\[
2\delta + \max(c,e^{-tcA})+b+8 \exp(-1/(32t)) < \frac{1}{24}\,
\]
we have 
\[
\p{\sum_{i \in I} M_i \ge A/24} \le 2e^{-200A}\, .
\]
It can be checked that taking $\delta=b=c=1/100$ does the job when $t > 100 \log 100/A$ (so that $\max(c,e^{-tcA})=1/100$)
and $t$ is sufficiently small (it is in order to satisfy these simultaneously that we require a lower bound on $A$). 
This completes the proof. 
\end{proof}

\begin{proof}[Proof of Fact~\ref{fact:massdoub}]
Let $\Z/2 = \{y/2: y \in \Z\}$. 
Define the event 
\[
E=\{\max\{ |X_i(r)|, i \ge 1, 0 \le r \le s+10^5/N\} \le 3s\}\, .
\] 
Any unit interval $[x-1/2,x+1/2]$ is covered by at most two intervals from $\{[y-1/2,y+1/2]: y \in \Z/2\}$. 
It follows that on $E$, if $\tau_k< s$ but $\sup\{z(t), t \in [\tau_k,\tau_k+10^5/N]\} > 10N$ then there is $y$ with $y \in [-3s,3s]\cap \Z/2$ such that 
\[
\sup_{t \in [\tau_k,\tau_k+10^5/N]}\sum_{\{i:|X_i(t)-y| <1/2\}}M_i(t) > 5N. 
\]
When $k < I$ we have $\tau_k < s$, so 
\begin{align}	\label{eq:y_bound}
	& \p{\sup_{t \in [\tau_k,\tau_k+10^5/N]} z(t) > 10N,z(\tau_k)\le N,k < I} \nonumber\\
\le 	& \p{E^c} + \sum_{y \in [-3s,3s]\cap\Z/2} \p{\sup_{t \in [\tau_k,\tau_k+10^5/N]}z(t,y) > 5N,z(\tau_k) \le N}.
\end{align}
Our bound on the above summands works identically for each $y \in [-3s,3s] \cap \Z/2$; we explain it for $y=0$ to avoid notational overload. 
So we wish to bound 
\[
\p{\sup_{t \in [\tau_k,\tau_k+10^5/N]}z(t,0) > 5N,z(\tau_k) \le N}.
\]

Our strategy is as follows:  we use Corollary~\ref{cor:decrease_2} to show that with high probability, for all $t \in [\tau_k,\tau_k+10^5/N]$ the total contribution to $z(t,0)$ from descendants of particles with $|X_i(\tau_k)| \le 3/2$ is at most $4N$. We then use Proposition~\ref{prop:mass_travel} to show that with high probability the contribution to $z(t,0)$ from descendants of further-off particles decreases quadratically (as a function of $|X_i(\tau_k)|$); since the quadratic series converges, this implies a bound on the total contribution from far-off particles. We now proceed to details.

For $n \in \Z$ let 
\[
Y_n = \sup_{t \in [\tau_k,\tau_k+10^5/N]} \sum_{\{i: |X_i(t)| \le 1/2, |X_{i,t}(\tau_k)-n| \le 1/2\}} M_i(t); 
\]
$Y_{n}$ counts the greatest contribution at any time $t \in [\tau_k,\tau_k+10^5/N]$, to the mass near $0$ from particles that at time $\tau_k$ are near $n$. We clearly have 
\begin{equation}\label{eq:z_local_contrib}
\sup_{t \in [\tau_k,\tau_k+10^5/N]} z(t,0) \le \sum_{n \in \Z} Y_n\, .
\end{equation}

As sketched above, we bound the sum in two parts: the contribution from $Y_{-1},Y_0$ and $Y_1$ is handled separately from the rest, and we do this first. 
Note that since masses decrease with time,
\[
Y_{-1}+Y_0+Y_1 \le \sup_{t \in [\tau_k,\tau_k+10^5/N]}\sum_{\{i:|X_{i,t}(\tau_k)| \le 3/2\}} M_{i,t}(\tau_k). 
\]
If $z(\tau_k) \le N$ then $\sum_{\{i: |X_i(\tau_k)| \le 3/2\}} M_i(\tau_k) \le 3N$ so, by 
Corollary~\ref{cor:decrease_2} and the strong Markov property, 
\begin{align*}
\p{Y_{-1}+Y_0+Y_1> 4N, z(\tau_k) \le N} & \le 2\pran{2^{1+1/3}(1-e^{-10^5/N})^{1/3}}^{3N} \\
					& \le (20^5/N)^{N}\, .
\end{align*}

Now consider $n \in \Z$ with $|n| \ge 2$, and assume by symmetry that $n > 0$. 
If $|X_i(t)| \le 1/2$ but  $|X_{i,t}(\tau_k)+n| \le 1/2$ then $X_i(t)-X_{i,t}(\tau_k) \ge n-1$. Assuming 
$z(\tau_k) \le N$, in particular we have $z(\tau_k,-n)\le N$. Furthermore, 
\[
Y_{-n} \le \sup_{t \in [\tau_k,\tau_k+10^5/N]} \sum_{\{i: X_i(t)-X_{i,t}(\tau_k) > n-1\}} M_{i,t}(\tau_k)\, .
\]
When $n \ge 2$, applying Corollary~\ref{cor:mass_travel} with $t_0 = 10^5/N$, $L=n-1$, $v=1/(20(n-1)^2)$ and $x = N$, we then obtain that 
\begin{align} \label{Y_far_bound}
\p{Y_{-n} > \frac{N}{10(n-1)^2}, z(\tau_k) \le N} & \le 
40(n-1)^2 \exp\left(\frac{10^5}{N}-\frac{N(n-1)^2}{2\cdot 10^5}\right)\, \notag \\
&  < 
\exp\left(-\frac{N(n-1)^2}{3 \cdot 10^5}\right) \notag \\
&\le  
s^{-10(n-1)^2}\, .
\end{align}
The final inequality holds since $N=N(s)=10^7 \log s$; the second inequality holds provided $N$ is sufficiently large. We emphasize that once $N$ is large enough the inequality holds for {\em all} $n \ge 2$. Note that by symmmetry the same bound also holds for $Y_n$. 

Using (\ref{eq:z_local_contrib}) and the two preceding probability bounds (and the fact that 
$(1/10)\sum_{|n| \ge 2} (n-1)^{-2} = \pi^2/30 < 1$), we thus have 
\begin{align}
& \p{\sup_{t \in [\tau_k,\tau_k+10^5/N]}z(t,0) > 5N,z(\tau_k) \le N}\nonumber\\
 \le & \p{Y_{-1}+Y_0+Y_1> 4N, z(\tau_k) \le N} + 
 \sum_{\{n \in \Z: |n| \ge 2\}} \p{Y_n \ge \frac{N}{10(n-1)^2},z(\tau_k) \le N} \nonumber\\
< &  \pran{\frac{20^5}{N}}^{N} + \sum_{|n| \ge 2} s^{-10(n-1)^2} \nonumber \\ 
< & \pran{\frac{20^5}{N}}^{N} + 4s^{-10}, \label{eq:for_nextfact}
\end{align}
where the last inequality holds for $s$ sufficiently large.
The same argument yields the same bound with $z(t,y)$ in place of $z(t,0)$, and (\ref{eq:y_bound}) then gives
\begin{align*}
	& \p{\sup_{t \in [\tau_k,\tau_k+10^5/N]} z(t) > 10N,z(\tau_k)\le N,k < I} \\
\le 	& \p{E^c} + 12s\cdot \pran{\pran{\frac{20^5}{N}}^{N} + 4s^{-10}}\\
\le 	& \p{E^c} + s^{-8}\, ,
\end{align*}
the latter bound holding for $s$ large, since $N=10^7 \log s$. To conclude, we use the fact that 
\[
\Cprob{ \max\{|X_i(s+10^5/N)|, i \ge 1\} \ge 3s}{E^c} \ge \frac{1}{2}, 
\]
which follows by considering the stopping time $\tau = \inf\{r: \max\{|X_i(r)|, i \ge 1\} \ge 3s\}$ and using symmetry. 
This yields  
\begin{align}
\p{E^c} & \le 2 \p{\max\{|X_i(s+10^5/N)|, i \ge 1\} \ge 3s} \nonumber\\
& \le 4\E{\#\{i: X_i(s+10^5/N) \ge 3s\}} \nonumber\\
& = 4e^s \p{N(0,s+10^5/N) \ge 3s} \nonumber\\
& \le e^{-3s} \label{eq:ebound}\\
& < s^{-8}, \nonumber 
\end{align}
where the last two inequalities hold for $s$ sufficiently large.
\end{proof}
\begin{proof}[Proof of Fact~\ref{fact:masslow}]
The proof has aspects which will be familiar from the previous proof; we describe these first. We recycle the event $E$ from the preceding proof. Note that on $E\cap \{k < I\}$ we have 
\[
z(\tau_k+10^5/N) \le 2 \sup_{y \in [-3s,3s] \cap \Z/2} z(\tau_k+10^5/N,y)\, ,
\]
so 
\begin{align}
& \p{z(\tau_k+10^5/N) \ge N-1, z(\tau_k) \le N, k < I, E}\nonumber\\
\le & 
\sum_{y \in [-3s,3s] \cap \Z/2} \p{ z(\tau_k+10^5/N,y) > \frac{N-1}{2}, z(\tau_k) \le N}. \label{eq:zbounds}
\end{align}
We once again focus on the case $y=0$ for notational simplicity. We write 
\[
Z_n = \sum_{\{i: |X_i(\tau_k+10^5/N)|<1/2, |X_{i,\tau_k+10^5/N}(\tau_k)-n|< 1/2\}} M_i(\tau_k+10^5/N)\, .
\]
The indices of summation correspond to particles with position near $0$ at time $\tau_k+10^5/N$, whose time $\tau_k$ ancestor had position near $n$. We have 
\[
z(\tau_k+10^5/N,0) \le \sum_{n \in \Z } Z_n. 
\]
Now similarly to the argument leading to (\ref{Y_far_bound}), apply Corollary~\ref{cor:mass_travel} with $t_0 = 10^5/N$, $L=n-1$, $v=1/(40(n-1)^2)$ and $x = N$ to bound $Z_n$ for $|n|\ge 2$. 
%
We obtain that for $s$ sufficiently large (since $(1/20)\sum_{|n| \ge 2} (n-1)^{-2} = \pi^2/60 < 1/2$) 
\begin{align}
& \p{z(\tau_k+10^5/N,0) \ge \frac{N-1}{2}, z(\tau_k) \le N} \nonumber\\
\le & \p{Z_{-1}+Z_0+Z_1 \ge \frac{N}{4}, z(\tau_k) \le N} + 
 \sum_{\{n \in \Z : |n| \ge 2\}}  \p{Z_n \ge \frac{N}{20(n-1)^2},z(\tau_k) \le N} \nonumber\\
\le  & \p{Z_{-1}+Z_0+Z_1 \ge \frac{N}{4}, z(\tau_k) \le N} + 4s^{-10}\, .\label{eq:firstbound}
\end{align}
We now bound $Z_{-1}+Z_0+Z_1$ from above by the {\em total mass at time $\tau_k+10^5/N$ of individuals whose time-$\tau_k$ ancestor lies in $[-3/2,3/2]$}. More precisely, recall that $X_{i,t}(s)$ is the (location of) the time-$s$ ancestor of $X_i(t)$, and write
\[
D_{\ell} = \sum_{\{i: X_{i,\tau_k+10^5/N}(\tau_k) \in [\ell/2,(\ell+1)/2]\}}M_i(\tau_k+10^5/N). 
\]
Then 
\[
Z_{-1}+Z_0+Z_1 \le  \sum_{\ell \in [-3,2]\cap \Z} D_\ell \, .
\]
This holds because the time-$\tau_k$ ancestors of particles counted by $Z_{-1}+Z_0+Z_1$ all lie in $[-3/2,3/2] = \bigcup_{\ell \in [-3,2] \cap \Z} [\ell/2,(\ell+1)/2]$. The bound may be strict because particles counted by $Z_{-1}+Z_0+Z_1$ are additionally required to lie near $0$ at time $\tau_k+10^5/N$. 

Bounding each of the summands $D_\ell$ by the largest summand, we then have 
\[
Z_{-1}+Z_0+Z_1 \le  6\max_{\ell \in [-3,2]\cap \Z} D_\ell\, , 
\]
so 
\begin{align*}
& \p{Z_{-1}+Z_0+Z_1 \ge \frac{N}{4}, z(\tau_k) \le N} \\
\le & 
 6\max_{\ell \in [-3,2]\cap \Z} \p{D_\ell > \frac{N}{24}, z(\tau_k)\le N}
\end{align*}
The final probabilities are not hard to bound: if $D_\ell$ hearkens from a total time-$\tau_k$ mass which is very small then at time $\tau_k+10^5/N$ it is still rather small by 
Corollary~\ref{cor:decrease_2}. On the other hand, if the aggregate mass of its time-$\tau_k$ ancestors was larger (but still at most $N$) then by Lemma~\ref{lem:decrease_1}, at time $\tau_k+10^5/N$ that ancestral population has lost most of its mass. 

More precisely, since $M_i(\tau_k+10^5/N)\le M_{i,\tau_k+10^5/N}(\tau_k)$ for each $i$, by Corollary~\ref{cor:decrease_2} and the strong Markov property, 
\[
\Cprob{D_\ell > \frac{N}{24}}{\sum_{\{j: X_j(\tau_k) \in [\ell/2,(\ell+1)/2]\}}M_j(\tau_k) \le N/48} \le 2(4(1-e^{-10^5/N}))^{N/48} \le 2 \pran{\frac{20^5}{N}}^{N/48}. 
\]
Now assume that $N=N(s)=10^7\log s > 48 C$, where $C$ is the constant from Lemma~\ref{lem:decrease_1}. By that lemma, since $10^5/N>500/(N/48)$,
\[
\Cprob{D_\ell > \frac{N}{24}}{\sum_{\{j: X_j(\tau_k) \in [\ell/2,(\ell+1)/2]\}}M_j(\tau_k) \in [N/48,N]} \le 2e^{-200N/48} < e^{-4N}, 
\]
the latter inequality for $N=N(s)$ sufficiently large. 
This bound holds for each $\ell \in [-3,2]\cap \Z$. Under the assumption that $z(\tau_k)\le N$, one of the conditions in the above conditional probabilities must occur. It follows that 
\begin{align*}
 6\max_{\ell \in [-3,2]\cap \Z} \p{D_\ell > \frac{N}{24}, z(\tau_k)\le N} 
 \le 6\max\pran{2\pran{\frac{20^5}{N}}^{N/48},e^{-4N}}\, ,
\end{align*}
so for $N$ sufficiently large 
\[
\p{Z_{-1}+Z_0+Z_1 \ge \frac{N}{4}, z(\tau_k) \le N} \le 6e^{-4N}= 6s^{-4\cdot 10^7}\, .
\]
Combined with (\ref{eq:firstbound}) this gives 
\[
\p{z(\tau_k+10^5/N,0) \ge \frac{N-1}{2}, z(\tau_k) \le N} < 5s^{-10}. 
\]
The same bound holds for each $z(\tau_k+10^5/N,y)$, so 
using (\ref{eq:zbounds}) and the bound $\p{E^c} \le e^{-3s}$ from (\ref{eq:ebound}), for $s$ large we obtain 
\[
\p{z(\tau_k+10^5/N) \ge N-1, z(\tau_k) \le N, k < I} \le 60s^{-9}+e^{-3s} < s^{-8}\, .
\]

The proof is almost complete; to finish it off we need to deal with the event $\{z(\tau_k) \le N\}$ in the preceding probability. 
To do so we use induction. 
First, for $s$ large, since $N=N(s)=10^7\log s$ and $\tau_0=\inf\{t:z(t) \ge N-1\}$, then 
$z(\tau_{0}^-)\le N-1$. It follows that
$z(\tau_{0}) \le z(\tau_{0}^-)+1 \le N$ (this was explained in the proof of Proposition~\ref{prop:massbound}), so when $k=0$ we have 
\[
\p{z(\tau_k+10^5/N) \ge N-1, k < I} = \p{z(\tau_k+10^5/N) \ge N-1, z(\tau_k) \le N, k < I} \le s^{-8}\, .
\]
For larger $k$, 
similarly
if $z(\tau_{k-1}+10^5/N) \le N-1$ then $z(\tau_{k}) \le z(\tau_{k}^-)+1 \le N$. We thus have 
\begin{align*}
\p{z(\tau_k+10^5/N) \ge N-1, k < I} & \le 
\p{z(\tau_k+10^5/N) \ge N-1, z(\tau_k) \le N, k < I} \\
& \quad + \p{z(\tau_k) > N,k < I} \\
& \le s^{-8} + \p{z(\tau_{k-1}+10^5/N) \ge N-1,k-1 < I}, 
\end{align*}
so by induction and the hypothesis that $k \le Ns/10^5$, 
\[
\p{z(\tau_k+10^5/N) \ge N-1, k < I} \le (k+1) \cdot s^{-8} 
< N s^{-7} < s^{-6}. \qedhere
\]
\end{proof}

\section{Lower bound}\label{sec:lower} 
The next proposition restates the second inequality of Theorem~\ref{thm:main}. Recall that 
$c^*=3^{4/3}\pi^{2/3}/2^{7/6}$. 
\begin{prop}\label{prop:lb}
For any $m  \in (0,1)$, almost surely 
\[
\liminf_{t \to \infty} \frac{\sqrt{2} t-D(t,m)}{t^{1/3}} \le c^*. 
\]
\end{prop}
Given a function $f:[0,\infty) \to \R$, for $t \ge 0$ let $I(t,f) = \{ i \ge 1: \forall s \in [0,t], X_{i,t}(s) \ge f(s)\}$ be the indices of particles 
whose ancestral trajectory stays above $f$ up to time $t$. Note that $|I(t,f)|$ is decreasing in $t$: if a trajectory stays above $f$ to time $t$ then it also stays above $f$ to time $t'<t$. It follows that 
$\p{\forall t, I(t,f)\ne\emptyset} = \lim_{t \to \infty} \p{I(t,f) \ne \emptyset}$, and this is a decreasing limit. 
We will use the following result of Roberts \cite{matt_cmd}.
\begin{lem}[\cite{matt_cmd}, Theorem 1] \label{lem:matt}
Let $g(t) = \sqrt{2} t - c^* t^{1/3} + c^* t^{1/3}/\log^2(t+e)-1$. Then 
\[
\lim_{t \to \infty} \p{ I(t,g) \ne \emptyset} = p^* > 0
\]
\end{lem}
The idea of the proof of Proposition~\ref{prop:lb} is that if the density is always low beyond $g$ then a particle staying beyond $g$ will have reasonably large mass at time $t$; the lemma guarantees that such a particle has a reasonable chance $p^*$ of existing. 
The next corollary implies that at the cost of a constant shift of the function $g$, we may increase $p^*$ as close to one as we like. For $c \in \R$ write $g-c$ for the function with $(g-c)(x) = g(x)-c$. 
\begin{cor} \label{C_star_cor}
Let $C^* = \inf\{c: \forall t, I(t,g-c) \ne \emptyset\}$. Then almost surely $C^* < \infty$. 
\end{cor}
\begin{proof}
The proof technique is sometimes called an amplification argument. 
Consider the $n(t)\approx e^t$ independent copies of the BBM rooted at time-$t$ particles, the $i$'th copy having initial individual at position $X_i(t)$. Suppose the ``translate by $X_i(t)$'' of the event from Lemma~\ref{lem:matt} occurs in the $k$'th copy; more precisely, suppose that for all $t' \ge t$ there is a descendant $X_j(t')$ of $X_k(t)$ such that for all $s \in [t,t']$, 
\[
X_{j,t'}(s) - X_k(t) \ge g(s-t) \ge g(s)-\sqrt 2 t - c^* t^{1/3}. 
\]
For $s \le t$ we also have 
\[
X_{j,t'}(s) \ge  \inf_{i \ge 1} X_i(s) \ge \inf_{s \le t}\inf_{i \ge 1} X_i(s) \ge g(s) + \inf_{s \le t}\inf_{i \ge 1} X_i(s) - \sup_{s \le t} g(s)\, ,
\]
so in this case 
\[
C^*\le -\inf_{s \in [0,t]}\inf_{i \ge 1} X_i(s) + \sqrt 2 t + c^* t^{1/3}. 
\]
By the branching property (i.e. the independence of the trajectories emanating from each of the particles $(X_i(t),i \ge 1)$), it follows that 
\begin{align}\label{eq:cor_3bd}
\p{C^* \le 3t + \sqrt 2 t + c^* t^{1/3}} & \le \p{n(t) \le 2^t} + \p{\inf_{s \in [0,t]}\inf_{i \ge 1} X_i(s) \le -3t} + (1-p^*)^{2^t}, 
\end{align}
where $p^*$ is the constant from Lemma~\ref{lem:matt}. Since $n(t)$ is Geom$(e^{-t})$ we have $\p{n(t) \le 2^t} \le (2/e)^t$. 
Finally, let $\sigma = \inf\{s: \inf_{i \ge 1} X_i(s) \le -3t\}$, so $\inf_{s \in [0,t]}\inf_{i \ge 1} X_i(s) \le -3t$ if and only if $\sigma < t$. 
Considering the descendants of the first individual to reach position $-3t$, by symmetry we have 
\[
\Cprob{\inf_{i \ge 1} X_i(t) \le -3t}{\sigma < t} \ge \frac{1}{2}\, ,
\]
so 
\[
\p{\sigma < t} \le 2 \p{\inf_{i \ge 1} X_i(t) \le -3t} \le 2e^t \p{N(0,t) \le -3t} \le e^{-7t/2}\, .
\]
These bounds and (\ref{eq:cor_3bd}) then yield 
\[
\p{C^* \le 3t + \sqrt 2 t + c^* t^{1/3}} \le (2/e)^t + e^{-7t/2} + (1-p^*)^{2^t}\, .
\]
This can be made arbitrarily small by taking $t$ large. 
\end{proof}
In order to prove Proposition~\ref{prop:lb}, we require one final lemma which shows that a small mass will quickly increase to form some region of constant density within a constant distance.
\begin{lem}\label{lem:mass_to_one}
For all $\eps > 0$ and $m \in (0,1)$ there is $C > 0$ such that for all $\bx \in \R^k$ and $\bm \in (0,1]^k$, if $z := \sum_{\{i:|x_i|< 1\}} m_i > 0$ then 
\begin{equation}\label{eq:mass_21}
\psub{\mathbf{x},\mathbf{m}}{\exists t \in [0,C(1+\log(1/z))],x \in [-C,C] : \zeta(t,x) \ge m} \ge 1-\eps\, .
\end{equation}
\end{lem}
To prove the lemma we use the following fact, whose proof is left to the reader. 
\begin{fact}\label{fact:strip_growth}
For all $\eps > 0$, there are $t_0=t_0(\eps)$ and $c=c(\eps)>0$ such that 
\begin{equation}\label{eq:hhk}
\p{\forall t \ge t_0, \#\{i: \forall s \in [0,t],|X_{i,t}(s)|<c\} \ge (e-\eps)^t} > 1-\eps. 
\end{equation}
\end{fact}
One straightforward way to prove the fact is as follows. First show that $p(c)$, the survival probability of branching Brownian motion with absorbing boundaries at $-c$ and $c$, started from the origin, satisfies $p(c) \to 1$ as $c \to \infty$. Then use a suitable branching process approximation. 
As an aside, we note the very nice recent work \cite{hhl} on the asymptotics of this survival probability for $c$ near the critical width $\hat{c}$ below which $p(c)=0$.
\begin{proof}[Proof of Lemma~\ref{lem:mass_to_one}]
The claim is clearly true if $z \ge m$, and we hereafter assume $z \in (0,m)$. We also assume $\eps$ is small enough that $(e-\eps)e^{-m}(1-\eps^{1/2}) > (1+\eps)$; this can only make our job harder. 

By relabelling, we may assume that for some $1 \le k' \le k$ we have $|x_i| < 1$ for $1 \le i \le k'$ and $|x_i| > 1$ for $i > k'$, so that $z = \sum_{1 \le i \le k'} m_i$. We also assume $x_1,\ldots,x_k$ are ordered so that $(m_i,1 \le i \le k')$ is decreasing. 

For $1 \le i \le k'$ let $J_i(t)$ index the time-$t$ descendants of $x_i$ whose trajectory stays fairly near the origin, i.e., 
\[
J_i(t) = \{\ell \ge 1: j_{\ell,t}(0)=i, |X_{\ell,t}(s)-x_i|<c~ \forall s \in [0,t]\}\, ,
\]
where $c$ is chosen as in Fact~\ref{fact:strip_growth}. By that fact, we then have for $t_0=t_0(\eps)$
\[
\p{\forall t \ge t_0, |J_i(t)| \ge (e-\eps)^t} > 1-\eps. 
\]
For $1 \le n \le k'$ write 
\[
S_n = \#\{1 \le i \le n: \forall t \ge t_0(\eps), |J_{i}(t)| \ge (e-\eps)^t\}. 
\]
Then $S=(S_n)_{1 \le n \le k'}$ stochastically dominates a random walk with Bernoulli$(1-\eps)$ steps. It follows by a ballot-type theorem (\citep[Corollary 11.17]{kallenberg}, for example, is sufficient for our needs) that for any $A>1$, 
\begin{equation}\label{eq:ballot}
\p{\exists n \le k': S_n < (1-A\eps) n} < A^{-1}.
\end{equation}

We hereafter assume $t \ge t_0(\eps)$. Now suppose that $\zeta(s,x) < m$ for all $s \le t$ and $|x| \le c+1$. Then for each $1 \le i \le k'$, for all $j \in J_i(t)$, $M_j(t) \ge m_i \cdot e^{-mt}$, so 
\[
\sum_{1 \le i \le k'} \sum_{j \in J_i(t)} M_j(t) 
\ge e^{-mt} \sum_{1 \le i \le k'} m_i |J_i(t)| 
\ge e^{-mt} (e-\eps)^t \sum_{1 \le i \le k'} m_i \I{|J_i(t)| \ge (e-\eps)^t}\, .
\]
Since the masses $m_i$ are decreasing in $i \in \{1,2,\ldots,k'\}$, if $S_n > (1-\eps^{1/2}) n$ for all $n \le k'$ then it follows that
\[
\sum_{1 \le i \le k'} \sum_{j \in J_i(t)} M_j(t) \ge e^{-mt} (e-\eps)^t \sum_{i=1}^{k'} m_i(1-\eps^{1/2}) = e^{-mt} (e-\eps)^t (1-\eps^{1/2}) \cdot z .
\]
By our assumption on $\eps$, we have  
$e^{-m}(e-\eps) > (1+\eps)/(1-\eps^{1/2}) > 1+2\eps$, so this gives 
\[
\sum_{j: |X_j(t)| < c+1} M_j(t) \ge (1+2\eps)^{t-1}z > c+2\, ,
\]
the last inequality provided that $t\ge 1+\log_{1+2\eps} ((c+2)/z)$. Since $[-c-1,c+1]$ can be covered by $\lfloor c+2 \rfloor $ intervals of radius $1$, we see that in this case there is $x$ with $|x| \le c+1$ such that $\zeta(t,x) > 1$. 

To sum up: assuming the random walk $S$ behaves, and that $t \ge t_0(\eps)$ and $t\ge 1+\log_{1+2\eps} ((c+2)/z)$, either $\zeta(s,x) \ge m$ for some $s\le t$ and $|x| \le c+1$, or else $\zeta(t,x) > 1$ for some $x$ with $|x| \le c+1$. By taking $A=\eps^{-1/2}$ in (\ref{eq:ballot}) and choosing $C=C(\eps)$ appropriately, we obtain 
\[
\psub{\mathbf{x},\mathbf{m}}{\exists s \in [0,C(1+\log(1/z))],x \in [-C,C]: \zeta(t,x) \ge m} \ge 1-\eps^{1/2}. \qedhere
\]
\end{proof}
We are now ready for the final proof of the paper. 
\begin{proof}[Proof of Proposition~\ref{prop:lb}]
Fix $m \in (0,1)$. Let $Z=2\cdot 10^8$ and 
\[t^* = \inf\{r \ge 0: \forall t\ge r, \sup_{s \in [0,t]} \sup_{x \in \R} \zeta(s,x) \le Z \log t\},\]
and note that $t^* < \infty$ almost surely by Proposition~\ref{prop:massbound} and the first Borel-Cantelli lemma. 

Fix $\eps > 0$ and by Corollary \ref{C_star_cor} choose $L>1$ large enough that $\p{\max(C^*,t^*) \ge L} <\eps$. 
Fix $t$ much larger than $L$ (so that $\log \log t > L$, say). 

Let $\sigma = \inf\{s \ge t^{1/4}: D(s,1/t) \ge g(s)-C^*-1\}$. We first suppose that $\sigma > t$, so that for all $s \in [t^{1/4},t]$ we have $D(s,1/t) < g(s)-C^*-1$. Let $i^*$ be such that $X_{i^*,t}(s) \ge g(s)-C^*$ for all $s \in [0,t]$; such $i^*$ exists by the definition of $C^*$. If $t^* \le L<t$ then we have 
\begin{align*}
-\log M_{i^*}(t)	& = \int_0^t \zeta(s,X_{i^*,t}(s)) \mathrm{d}s \\
			& \le \int_0^{t^{1/4}} \zeta(s,X_{i^*,t}(s))\mathrm{d}s + \int_{t^{1/4}}^t \frac{1}{t}\mathrm{d}s\, \\
			& \le Zt^{1/4}\log t + 1\, ,
\end{align*}
the last bound because when $t \ge t^*$ the integrand is at most $Z\log t$. 

Let $C=C(\eps,m)$ be the constant from Lemma~\ref{lem:mass_to_one}. Then by that lemma (applied with $z=M_{i^*}(t) \ge \exp(-1-Zt^{1/4}\log t)$) and the Markov property, given that $\{t^* \le t\}$, with probability at least $1-\eps$ there is $s\in (t,t+C(2+Zt^{1/4}\log t))$ and $x$ with $|x| \le C$ such that $\zeta(s,X_{i^*}(t)+x) \ge m$. 
If this occurs, and additionally $C^* \le L$ we have 
\[
D(s,m) \ge  X_{i^*}(t)-C\ge g(t)-C^* -C \ge g(s)-s^{1/4}\log^2 s\, ,
\]
the last bound holding for all $t$ sufficiently large since $s -t \le C(2 + Z t^{1/4} \log t)$, and for $s$ and $t$ large we have $g(s) -g(t) =O(s-t)$. We thus have 
\begin{align}
	& \probC{\exists s \ge t: D(s,m) \ge g(s)-s^{1/4}\log^2 s}{\sigma > t} \nonumber\\
\ge	&  \probC{\max(C^*,t^*)<L,~\exists s \ge t: D(s,m) \ge g(s)-s^{1/4}\log^2 s}{\sigma > t}\nonumber\\
	& 
	- \probC{\max(C^*,t^*) \ge L}{\sigma > t} \nonumber\\
\ge	&1-\eps-\probC{\max(C^*,t^*) \ge L}{\sigma > t}\, .\label{eq:sigma_ge_t}
\end{align}

Next suppose that $\sigma \le t$. Apply the strong Markov property at time $\sigma$, and apply Lemma~\ref{lem:mass_to_one} just as above (but with a starting mass in $[D(\sigma ,1/t)-1,D(\sigma,1/t)+1]$ of at least $1/t=e^{-\log t}$ rather than $e^{-1-Zt^{1/4}\log t}$). We obtain that 
with probability at least $1-\eps$ there is $s \in (\sigma,\sigma+C(1+\log t))$ such that 
\[
D(s,m) \ge g(\sigma)-C-C^* \ge g(s)-\log^2 s\, ,
\]
the last bound holding for $t$ sufficiently large since $s-\sigma \le  C(1+\log t)$ and $\log t \le 4\log \sigma \le 4\log s$, and under the assumption $C^* \le L$. 

Since $\sigma \ge t^{1/4}$ and $\log^2 s < s^{1/4} \log^2 s$, it follows that 
\begin{align*}
	& \probC{\exists s \ge t^{1/4}: D(s,m) \ge g(s)-s^{1/4}\log^2 s}{\sigma \le t} \\
\ge	& 1-\eps- \probC{\max(C^*,t^*) \ge L}{\sigma \le t}\, .
\end{align*}
Now combine this with (\ref{eq:sigma_ge_t}) using the law of total probability. 
We chose $L$ large enough that $\p{\max(C^*,t^*) \ge L} \le \eps$, so we obtain 
\[
\p{\exists s \ge t^{1/4}: D(s,m) \ge g(s)-s^{1/4} \log^2 s} > 1-2\eps\, .
\]
Finally, if $D(s,m) \ge g(s)-s^{1/4} \log^2 s$ then 
\[
\frac{\sqrt{2}s-D(s,m)}{s^{1/3}} \le c^* - \frac{c^*}{\log^2(s+e)} + \frac{1}{s^{1/3}} + \frac{\log^2 s}{s^{1/12}}\, ,
\]
which tends to $c^*$ as $s \to \infty$. 
\end{proof}
\section{Discussion and questions}\label{sec:conc}

\begin{itemize}
\item The analysis of the paper should carry through fairly straightforwardly to higher dimensions $\mathbb{R}^k$, provided we redefine $d(t,m)$ and $D(t,m)$ as
\[
d(t,m) = \min\{|x|: \zeta(t,x) < m\}, \quad D(t,m) =\max\{|x|: \zeta(t,x) > m\}\, .
\]
At time $t$, the density is then at least $m$ within the ball of radius $d(t,m)$ around $0$, and less than $m$ outside the ball of radius $D(t,m)$ around $0$. 
The proof of the lower bound is then the same as in Sections 4 and 5.
The proof of the upper bound requires ruling out the possibility that the modulus of a particle in the BBM stays ahead of a moving barrier $g$ even though it cannot have consistent displacement more than $g$ in any fixed direction.
In order for our proof techniques to carry over, this requires sample path estimates for $\mathrm{Bes}(k)$ processes analogous to the ones derived in this work for Brownian motion. We expect such estimates to hold for all $k \ge 1$, though verifying this may be technical.

\item We believe that Proposition~\ref{prop:lb} predicts the ``true'' front location, in that both $D(t,m)$ and $d(t,m)$ are typically at distance $o(t^{1/3})$ from $\sqrt{2}t-c^* t^{1/3}$ when $t$ is large. This is our justification for the remark in the final paragraph of Section~\ref{sec:intro}. 


\item In the same way as the KPP equation describes the evolution of multiplicative functionals of BBM \cite{mckean75}, it seems plausible that the model proposed in this work (or a related model) should be connected to an equation of the form 
\[
u_t = \frac{1}{2} u_{xx} - u(1-u) - \int_{\{y:|y-x|<1\}} u(t,y)\mathrm{d}y. 
\]
This equation has steady states at $0$ (unstable) and $1/2$ (stable), and is redolent of a family of ``non-local'' Fisher-KPP-type equations which was introduced \cite{britton89} to model populations in which aggregation can have both a competitive advantage (safety in numbers) and disadvantage (due to competition for resources). These equations have received substantial study  \cite{berestycki09,britton90,gourley00}; the survey \cite{volpert2009reaction} contains many further references, as well as perspective on the biological motivations for such study. 

If a probabilistic model for such an equation were found, it could yield new results on, e.g., the front propagation speed or temporal fluctuations of solutions to the above equation. Conversely, a glance at that literature suggests new probabilistic questions: for example, what if the effect of competition is described by a kernel $\kappa$, where $\kappa(|x-y|)$ describes the degree of competition for resources between individuals at spatial positions $x$ and $y$? In our model we took $\kappa(|x-y|) = \I{|x-y| \in (0,1)}$; a kernel which allows substantial long-range interaction might yield rather different dynamics.

\item As mentioned in the introduction, one may reasonably consider the mechanism for mass growth in our model -- both children inherit the mass of the parent -- nonphysical. More physically realistic (at least for amoebae) is for the children to each have half the mass of the parent. One must also then change the rules to allow for mass growth; a reasonable modification is to take 
\[
\overline{\zeta}(t,x) = \sum_{\{i:|X_i(t)-x| \le 1\}} M_i(s), 
\]
and 
\[
M_i(t) = \exp(\int_0^t (1-\overline{\zeta}(s,X_i(s))) \mathrm{d}s)\, .
\]
In other words, the mass of an individual can increase, when there is little nearby competition for resources -- but the larger particles get, the harder it is for them to sustain themselves. The key point is that $1$ is still a universal upper bound on the greatest mass of any particle. 

We conjecture that any lack of physical realism in our model is relatively insignificant for the long term behaviour, and more concretely that the front location behaves similarly in the two models. 
As partial evidence for this, we note that the analyses from Sections~\ref{sec:max_density} and~\ref{sec:lower} carry through essentially unchanged for the model described above.

The argument from Section~\ref{sec:ub}, however, breaks down, because a particle moving through an environment of constant density $m<1$ will have mass which does {\em not} decay exponentially, even when the loss of mass due to branching is taken into account. Instead, such a particle will (at large times) have a mass which is random and typically of order $\Theta(1-m)$. 

Because of this, the existing argument only establishes Proposition~\ref{upper_bound} in a highly weakened form, with the condition $m\ge 1$ rather than $m > 0$. (It is possible to do very slightly better, by considering a variable bound $m=m(t)$. One can then take $m(t) < 1$ if $1-m(t)$ decays sufficiently quickly, but the pain-to-gain ratio in writing down such an argument in detail does not seem favourable.) But $m \in (0,1)$ is the really meaningful region. Proving a genuine analogue of Proposition~\ref{upper_bound} for this model seems to us the only missing step to a proof of Theorem~\ref{thm:main} for the modified dynamics.

\item In the variant just described, one intriguing possibility is that there may now be particles with mass $\Theta(1)$ at large times. If there are, they will be found near the front, since that is where they can find food. Do they really exist? 

\item More generally, one may take 
\[
M_i(t) = \exp(\int_0^t (a-b\overline{\zeta}(s,X_i(s))) \mathrm{d}s)\, .
\]
This looks, heuristically, like some sort of spatial logistic growth \cite{etheridge,lambert}. It may be interesting to investigate what different behaviours can occur as the parameters $a$ and $b$ are varied. 
\end{itemize}

\section{Acknowledgements}
The first author thanks Jeremy Quastel for asking a stimulating question about branching random walk with cutoff \cite{MR3013924}, which planted the seed of the current paper, as well as Elie A\"id\'ekon, Henri Berestycki, Julien Berestycki, Nathana\"el Berestycki and James Martin, for valuable comments. 
The second author thanks Julien Berestycki and David Corlin-Marchand for interesting discussions.
Both authors thank the Newton Institute, where parts of this manuscript were written.  
                 
\appendix

\section{Estimates for the Upper Bound} \label{append_upper}
We first turn to the proof of Lemma~\ref{tailor_made_roberts}. 
The proof relies on the following sample path estimate for Brownian motion. 
\begin{lem} \label{roberts_lemma}
Suppose $f:[0,t]\rightarrow \R$ and $L:[0,t]\rightarrow (0,\infty)$ are twice continuously differentiable functions, with $f(0)<0$, $f(0)+L(0)>0$ and $f$ increasing. We assume that there exists a constant $Q>0$ such that 
\[
|L'(0)|L(0)+|L'(u)|L(u)+\int_0^u |L''(s)|L(s)\, ds +\int_0^u |f''(s)| L(s)\, ds -|L'(0) | f(0) \leq Q 
\]
for all $0\leq u\leq t$, which we call Assumption (A). Then there is a constant $M(Q)$ such that for $0\leq p\leq 1$,
\begin{align*} &\p{B(s)-f(s)\in (0,L(s))\forall s\leq u, B(u)-f(u)\in (pL(u),L(u))}\\
&\quad \leq M(Q)\exp(-\tfrac{1}{2}\int_0^u f'(s)^2 \, ds - \int_0^u \frac{\pi^2}{2L(s)^2}\, ds -p f'(u)L(u)-f'(0)f(0)+\tfrac{1}{2}\log L(u)).
\end{align*}
\end{lem} 
This result is obtained by combining Proposition 4 and Lemma 7 in \cite{matt_cmd} to cover the two cases $\int_0^u \frac{1}{L(s)^2}ds>1$ and $\int_0^u \frac{1}{L(s)^2}ds\leq 1$.
In order to apply Lemma~\ref{roberts_lemma}, 
we exploit the existence of a solution to an integral equation; such a solution is used for a related purpose in Section 3.4 of \cite{jaffuel2012}. 
\begin{lem}
For $c<c^*$, there exists a constant $\alpha>0$ such that the equation
\begin{equation} \label{l_defn}
l(s)=\alpha + c s^{1/3}-\frac{\pi^2}{2\sqrt{2}}\int_{0}^s \frac{1}{l(u)^2}\, du
\end{equation}
has a continuous solution on $[0,1]$ which is twice continuously differentiable on $(0,1)$ with $l(s)>0$ for all $s \in[0,1)$ and with $l(1)=0$.
\end{lem}
The lemma follows from Propositions~3.2 and~3.6(iii) of \cite{jaffuel2012}. More precisely, in those lemmas there is a variance term $\sigma$, and the value analogous to $c^*$ is $a_c=\tfrac{3}{2}(3\pi ^2 \sigma^2 )^{1/3}$. Taking $\sigma=1/\sqrt{2}$ yields the above formulation. 

\begin{proof}[Proof of Lemma~\ref{tailor_made_roberts}]
Fix $t>0$ large. Since the integral on the RHS of \eqref{l_defn} is non-negative, $l(s)\leq c+\alpha $ for all $s\in [0,1]$. 
Since $l(0)=\alpha>0$, we can fix $\beta \in(0,\min(\alpha^3/8,\alpha^3/(8c^3),1))$ sufficiently small that $l(s)\ge \alpha/2$ for $s <\beta$. 
Let 
\begin{equation} \label{L_defn}
L(s)= t^{1/3}\left(\frac{1+\beta }{u_t}\right)^{1/3}\,l\left(\frac{(s+\beta t) u_t}{t+\beta t}\right)
\end{equation}
 for $0\leq s \leq t$, where $u_t= \inf \{u\in [0,1]:l(u)\leq  2 t^{-1/12}\}$. Note that $u_t \to 1$ as $t \to \infty$.

We will prove that the lemma holds for above choice of $\beta$ and with the function $\Delta (s)=L(s)-Kt^{1/6}$, provided $K$ is sufficiently large. We must thus verify that $\Delta$ satisfies the requisite properties, and prove the bound (\ref{eq:tmr}). Write $f(s) = b(s)-Kt^{1/6} = \sqrt{2}s-c(s+\beta t)^{1/3}-Kt^{1/6}$; then to prove (\ref{eq:tmr}), it suffices to show that for $u \in [0,t]$ and all $x\in [0,L(u))$,
\begin{equation}\label{eq:toprove_fin}
\p{B(s) \in (f(s),f(s)+L(s)) \, \forall s\leq u, B(u)>f(u)+x}
\leq \exp (-u-t^{1/3}/K +\sqrt 2 (L (u)-x)).
\end{equation}
We establish this by applying Lemma~\ref{roberts_lemma} with the above functions $f$ and $L$. We next derive the properties of $f$, $L$, $L'$ and $L''$ which we require to do so. 

First, note that $f(0)=-c(\beta t)^{1/3}-Kt^{1/6}$ and $L(0)\ge t^{1/3}l((\beta u_t)/(1+\beta ))\ge \alpha t^{1/3}/2$ by our choice of $\beta$, so since $\beta<(\alpha/(2c))^3$ we have $L(0)+f(0)>0$ for $t$ sufficiently large.
For $t$ sufficiently large, $f$ is increasing on $[0,t]$, and $f$ and $L$ are twice continuously differentiable (since $l$ is $C^2$ on $(0,1)$).

We assume $t$ is sufficiently large that $(1+\beta)^{1/3}u_t^{-1/3}\le 2$. Then for $s\in [0,t]$, $L(s)\geq 2 (1+\beta )^{1/3} u_t^{-1/3}t^{1/4}\geq 2 t^{1/4}$ and $L(s)\leq (1+\beta )^{1/3} u_t^{-1/3}t^{1/3}(c+\alpha)\leq 2 (c+\alpha)t^{1/3}$, so for all $s\in [0,t]$,
\begin{equation} \label{L_bound}
2t^{1/4}\leq L(s)\leq 2(c+\alpha)t^{1/3}.
\end{equation}
Since $l$ is $C^2$ on $(0,1)$ and $l(s)>0$ for $s<1$, we can differentiate both sides of \eqref{l_defn} for $s\in (0,1)$ to give 
\begin{equation} \label{l_diff}
l'(s)=\frac{1}{3}c s^{-2/3}-\frac{\pi ^2}{2\sqrt 2}\frac{1}{l(s)^2}.
\end{equation} 
Hence $L$ is differentiable on $[0,t]$ with
\[
L ' (s)=t^{-2/3}\left(\frac{1+\beta }{u_t} \right)^{-2/3}\,l' \left(\frac{u_t(s+\beta t)}{t+\beta t} \right). 
\]
Also, for $\frac{u_t \beta }{1+\beta}\leq u \leq u_t$, by \eqref{l_diff} and the definition of $u_t$,
\begin{equation} \label{l_diff_bound}
|l'(u)|\leq \tfrac{1}{3}c (\beta u_t)^{-2/3}(1+\beta)^{2/3}+\tfrac{\pi^2}{8\sqrt 2}t^{1/6}\leq 2t^{1/6},
\end{equation}
for $t$ sufficiently large, so for all $s\in [0,t]$ we have 
\begin{equation} \label{L_diff_bound}
|L'(s)|\leq u_t^{2/3}(1+\beta)^{-2/3}t^{-2/3}2t^{1/6} \leq 2 t^{-1/2}.
\end{equation} 

This is a convenient moment to verify that the function $\Delta (s)=L(s)-Kt^{1/6}$ has the requisite properties. By \eqref{L_bound}, for $t$ sufficiently large, $\Delta:[0,t]\to [t^{1/4},2(c+\alpha)t^{1/3}]$. 
Also $\Delta (t)=L(t)-Kt^{1/6}< 2(1+\beta )^{1/3}u_t^{-1/3}t^{1/4}\le 4 t^{1/4} $ for $t$ sufficiently large. Finally, by \eqref{L_diff_bound}, $|\Delta '(s)|\le 2 t^{-1/2}\le 1$ for all $s\in [0,t]$, once again for $t$  sufficiently large. 


Proceeding with the proof of (\ref{eq:toprove_fin}), we now check that Assumption (A) holds for our choice of $f$ and $L$, for some constant $Q$ which does not depend on $t$. For $t$ sufficiently large, 
by \eqref{L_bound} and \eqref{L_diff_bound} we have 
 $\sup_{s \in [0,t]} |L'(s)L(s)| = O(t^{-1/6})$, and also  $|L'(0)f(0)| = O(t^{-1/6})$. 

By the definition of $L$ in~\eqref{L_defn}, for $s\in [0,t]$ we have
\[
|L'' (s)|L(s)= t^{-4/3} \left(\frac{1+\beta}{u_t}\right)^{-4/3}\left|l''\left( \frac{u_t (s+\beta t)}{t+\beta t}\right)\right| l\left( \frac{u_t (s+\beta t)}{t+\beta t}\right). 
\]
For $\frac{u_t \beta}{1+\beta}\leq u \leq u_t$ we also have $|l'(u)|\le 2 t^{1/6}$ by \eqref{l_diff_bound} and $2t^{-1/12} \le l(u)\le c+\alpha $ by the definition of $u_t$. By differentiating \eqref{l_diff} we obtain that, uniformly over $u$ in the above range, 
\[
|l''(u)|l(u)\leq \frac{2}{9}c u^{-5/3}l(u)+\frac{\pi^2}{\sqrt 2}\frac{|l'(u)|}{l(u)^2} = O(t^{1/3})\, , 
\]
and hence $\sup_{s \in [0,t]} |L''(s)|L(s)=O(t^{-1})$. Finally, $\sup_{s \in [0,t]} f''(s)=\sup_{s \in [0,t]} \frac{2}{9}c(s+\beta t)^{-5/3} = O(t^{-5/3})$, which with \eqref{L_bound} yields $\sup_{s \in [0,t]}|f''(s)| L(s)=O(t^{-4/3})$. Thus, Assumption (A) holds for some fixed constant $Q$ not depending on $t$.

Having verified the conditions of Lemma~\ref{roberts_lemma}, we now show that the bound from that lemma indeed implies (\ref{eq:toprove_fin}). 

For $0\leq u \leq t$, $f'(u)=\sqrt 2 - \tfrac{1}{3}c(s+\beta t)^{-2/3}$ so
\begin{align*}
-\tfrac{1}{2}\int_0^u f'(s)^2 \, ds&=-u+\sqrt 2 c(u+\beta t)^{1/3}-\sqrt 2 c (\beta t)^{1/3} +\tfrac{1}{3}c^2(u+\beta t)^{-1/3}-\tfrac{1}{3}c^2(\beta t)^{-1/3}\\
&\le -u+\sqrt 2 c u^{1/3}+O(t^{-1/3}). 
\end{align*}
Also by the definition of $L$ in \eqref{L_defn},
\[
\frac{\pi^2}{2}\int_0^u \frac{1}{L(s)^2}\, du =\frac{\pi^2}{2}t^{1/3}\left(\frac{1+\beta}{u_t}\right)^{1/3}\left(\int_{\frac{\beta u_t}{1+\beta}}^{\frac{(u+\beta t)u_t}{t+\beta t}}\frac{1}{l(s)^2}\, ds\right).
\]
We chose $\beta$ sufficiently small that for $s\leq \frac{\beta u_t}{1+\beta}$, $l(s)\ge \alpha/2$. Therefore
\[
 \int^{\frac{\beta u_t}{1+\beta}}_0\frac{1}{l(s)^2}\, ds \leq \frac{4}{\alpha^2}\frac{\beta u_t}{1+\beta}\le \frac{\alpha}{2},
 \]
since we also chose $\beta < \alpha^3/8$.
It follows by \eqref{l_defn} that
\begin{align*}
\frac{\pi^2}{2}\int_0^u \frac{1}{L(s)^2}\, du &\ge 
\sqrt 2 t^{1/3} \left(\frac{1+\beta}{u_t}\right)^{1/3}\left(\frac{\alpha}{2}+c \left( \frac{(u+\beta t)u_t}{t+\beta t}\right)^{1/3}- l \left(\frac{(u+\beta t)u_t}{t+\beta t} \right) \right)\\ 
&\ge (\sqrt 2/2) \alpha t^{1/3}+\sqrt 2 c u^{1/3}-\sqrt 2 L(u),
\end{align*}
where the second line follows by the definition of $L$ in \eqref{L_defn}.
Hence for $u\leq t$ and $p\in [0,1]$,
\begin{align} \label{roberts_bound}
&\exp(-\tfrac{1}{2}\int_0^u f'(s)^2 \, ds - \int_0^u \frac{\pi^2}{2L(s)^2}\, ds -p f'(u)L(u)-f'(0)f(0)+\tfrac{1}{2}\log L(u)) \notag \\
&\le \exp(-u-(\sqrt 2/2)\alpha t^{1/3}+\sqrt 2 (1-p)L(u)+\sqrt 2 c \beta^{1/3} t^{1/3} +O(t^{1/6}))\nonumber\\ 
&\le (M(Q))^{-1}\exp (-u-t^{1/3}/K +\sqrt 2 (1-p)L(u)). 
\end{align}
The last inequality holds for all large $t$, provided $K$ is sufficiently large that $(\sqrt{2}/2)\alpha-\sqrt{2}c\beta^{1/3}>1/K$; this is possible by our choice of $\beta$. 

Writing $p=x/L(u)$, then 
\begin{align*}
& \p{B(s) \in (f(s),f(s)+L(s)) \, \forall s\leq u, B(u)>f(u)+x}\\
= & \p{B(s)-f(s)\in (0,L(s))\forall s\leq u, B(u)-f(u)\in (pL(u),L(u))} \\
&\leq \exp (-u-t^{1/3}/K +\sqrt 2 (1-p)L(u))
\end{align*}
for some $K>0$ by Lemma \ref{roberts_lemma} and \eqref{roberts_bound}. This establishes (\ref{eq:toprove_fin}) and completes the proof. 
\end{proof}

We now turn to the proof of Lemma~\ref{lem:gb}, during which we will use the following fact. 
\begin{fact}\label{fac:gauss}
Let $(W(u),0 \le u \le 1)$ be either Brownian excursion or Brownian meander, and let $N$ be a standard Gaussian. Then 
\[
\p{\max_{u \le 1} W(u) \ge x} \le 4 \p{N \ge x/4}. 
\]
\end{fact}
\begin{proof} [Proof of Fact \ref{fac:gauss}]
Write $B$, $B^{\mathrm{me}}$, $B^{\mathrm{ex}}$ and $B^{\mathrm{br}}$ for Brownian motion, meander, excursion, and bridge, all of length one. 
In what follows, maxima are always over $u \in [0,1]$ even if this is not explicitly written. 

We have 
\[
\max B^{\mathrm{ex}} 
\eqdist
\max B^{\mathrm{br}} - \min B^{\mathrm{br}}
\le 
2\max |B^{\mathrm{br}}|
\eqdist
2\max_{u \le 1} |B(u)-uB(1)| \le 4\max |B|,
\]
so by the reflection principle, 
\[
\p{\max B^{\mathrm{ex}} \ge x} \le 
\p{\max|B| \ge x/4} \le 2 \p{\max B \ge x/4} = 4 \p{N \ge x/4}\, .
\]

Next, let $(Z(u),u \ge 0)$ be a standard Bessel process. 
Since $B^{\mathrm{me}}$ is Brownian motion conditioned to stay positive until time one, and $Z$ is Brownian motion conditioned to stay positive for all time, it follows straightforwardly that $\max_{u \le 1} Z(u)$ stochastically dominates $\max B^{\mathrm{me}}$. 
By Pitman's $2M-B$ theorem (see \cite{mp10}), 
we have 
\[
\max_{u \le 1} Z(u) \eqdist \max_{u \le 1} |B(u)-2\inf_{t \le u} B(t)| 
\le 3 \max |B|, 
\]
and it follows as above that 
\[
\p{\max B^{\mathrm{me}} \ge x} \le 
\p{\max |B| \ge x/3} \le 4 \p{N \ge x/3}\, .\qedhere
\]
\end{proof}
\begin{proof}[Proof of Lemma \ref{lem:gb}]
First, for a standard Gaussian $N$ we have $\p{N \ge c} \le \frac{1}{\sqrt{2\pi}}\frac{1}{c} e^{-c^2/2}$, for all $c > 0$. 
Using this bound, Fact~\ref{fac:gauss} and Brownian scaling, for each $i$ we obtain 
\[
\p{\max_{u \le t_i} X_i(u) \ge x} \le 
4\p{N \ge \frac{x}{4t_i^{1/2}}} \le \frac{1}{\sqrt{2\pi}}\frac{16t_i^{1/2}}{x} e^{-x^2/16t_i} < \frac{8t_i^{1/2}}{x} e^{-x^2/16t_i} \, .
\]
Provided that $x \ge 8T ^{1/2}$, a union bound then yields
\[
\p{\max_{i \ge 1} \max_{u \le t_i} X_i(u) \ge x}
\le \sum_{i \ge 1} e^{-x^2/16t_i}. 
\]
Finally, the function $f(a)=e^{-x^2/a}$ is convex for $a\in [0,x^2/2]$, and $f(0)=0$, so if $x^2 \geq 32T$ then for each $i$, $f(16t_i)\leq (t_i/T) f(16T)$. Hence 

\[
\p{\max_{i \ge 1} \max_{u \le t_i} X_i(u) \ge x}
\le  e^{-x^2/16T}. \qedhere
\]
\end{proof}

\section{Probability tail bounds}
We first state a Bernstein-type inequality due to Colin McDiarmid. 
\begin{thm}[\cite{mcdiarmid98}, Theorem 2.7] \label{thm:bern}
Let $X_1,\ldots,X_n$ be independent with $X_k - \e X_k \le 1$ for each $k$. Write $S_n=\sum_{k=1}^n X_k$ and 
fix $V \ge \V{S_n}= \sum_{k=1}^n \V{X_i}$. Then for any $c \ge 0$, 
\[
\p{S_n \ge \e S_n + c} \le e^c \cdot \pran{\frac{V}{V+c}}^{V+c} < \pran{\frac{eV}{c}}^c\, .
\]
\end{thm}
The first inequality is the heart of the theorem; the second is easy and is included to simplify an application of the theorem. 
The next lemma provides upper tail probability estimates for weighted geometric sums. 
\begin{lem}\label{lem:geom_ld}
Fix $\eps < 1/2$ and let $(G_i,i \ge 1)$ be i.i.d.\  $\mathrm{Geom}(1-\eps)$. For any $n$ and any non-negative real numbers $r_1,\ldots,r_n$ with $\max r_i /\sum r_i \le 1/V$, 
for all $\delta > 0$, 
\[
\p{\sum_{i=1}^n r_i G_i \ge (1+\delta)\sum_{i=1}^n r_i} \le 2(2^{1+\delta}\eps^{\delta})^V\, .
\]
\end{lem}
\begin{proof}[Proof of Lemma~\ref{lem:geom_ld}]
Let $\hat{G}_j=G_j-1$ and $p_j=r_j/\sum_i r_i$.  Then we must bound 
\[
\p{\sum_{i=1}^n p_i \hat{G_i} \ge \delta}\, ,
\]
under the assumption that $\max_i p_i \le 1/V$. First note that for $c$ such that $\eps e^c<1$,
\begin{align*}
\E{\exp\pran{cV\cdot \sum_{i=1}^n p_i \hat{G_i}}} & = \prod_i \frac{ 1-\eps}{1-\eps e^{cVp_i}}\, .
\end{align*}
For $c>0$, the latter product is maximized (subject to the constraints that $\max_i p_i \le 1/V$ and that $\sum_i p_i =1$) when 
$p_i=1/\lceil V \rceil$ for $\lceil V \rceil$ values of $i$ and $p_i=0$ otherwise. 
We thus obtain 
\[
\E{\exp(cV\cdot \sum_{i=1}^n p_i \hat{G_i})} \le \frac{(1-\eps)^{\lceil V \rceil}}{(1-\eps e^c)^{\lceil V \rceil}} 
\,. 
\]
For any non-negative random variable, $\p{X> \delta} \le e^{-c\delta V}\e{e^{cVX}}$; 
taking $e^c=(2\eps)^{-1}$ yields
\[
\p{\sum_{i=1}^n p_i \hat{G_i} \ge \delta} \le e^{-c\delta V} \cdot 
\frac{(1-\eps)^V}{(1-\eps e^c)^{V+1}} = 
2\pran{2(1-\eps)(2\eps)^{\delta}}^V < 2(2^{1+\delta}\eps^{\delta})^V\, .\qedhere
\]
\end{proof}

\small 


\normalsize
\vspace{-0.2cm}
\end{document}